\documentclass[12pt]{article}
\usepackage[margin=1in]{geometry}  
\usepackage{graphicx}              
\usepackage{amsmath}               
\usepackage{amsthm}
\usepackage{amssymb}
\usepackage[ansinew]{inputenc}
\usepackage{array}
\usepackage{amsxtra}
\usepackage{amstext}
\usepackage{latexsym}
\usepackage{dsfont}
\usepackage[all]{xy}
\usepackage{amsfonts}
\usepackage{eucal}
\usepackage{amscd}
\usepackage{multicol}
\usepackage[all]{xy}           
\usepackage{graphicx}
\usepackage{color}
\usepackage{colordvi}
\usepackage{xspace}

\usepackage{enumitem}  
\usepackage{calc}
\setlist{labelindent=1pt,itemsep=.5em}
\setlist[itemize]{leftmargin=1.2cm}
\setlist[enumerate]{itemindent=0em,leftmargin=1.2cm}
\setlist[enumerate,1]{label={\upshape(\roman*)}}

\usepackage{ifpdf}
\ifpdf
 \usepackage[colorlinks,final,hyperindex]{hyperref}
\else
\usepackage[colorlinks,final,hyperindex,hypertex]{hyperref}
\fi

\usepackage{cite}

\usepackage{authblk}

\allowdisplaybreaks

\makeatletter
\newcommand{\subjclass}[2][2020]{%
  \let\@oldtitle\@title%
  \gdef\@title{\@oldtitle\footnotetext{#1 \emph{Mathematics subject classification}: #2}}%
}
\newcommand{\keywords}[1]{%
  \let\@@oldtitle\@title%
  \gdef\@title{\@@oldtitle\footnotetext{\emph{Keywords}: #1.}}%
}
\makeatother

\newtheorem{thm}{Theorem}[section]
\newtheorem{cor}[thm]{Corollary}
\newtheorem{lem}[thm]{Lemma}
\newtheorem{prop}[thm]{Proposition}
\theoremstyle{definition}
\newtheorem{defn}[thm]{Definition}
\theoremstyle{remark}
\newtheorem{rmk}[thm]{Remark}
\theoremstyle{remark}

\newtheorem{ex}[thm]{Example}

\numberwithin{equation}{section}

\def\a{\alpha}
\def\b{\beta}

\begin{document}

\title{Bimodules and matched pairs of noncommutative BiHom-(pre)-Poisson algebras}
\author{Ismail Laraiedh \thanks{Departement of Mathematics, Faculty of
Sciences, Sfax University, BP 1171, 3000 Sfax, Tunisia. E.mail:
Ismail.laraiedh@gmail.com and Departement of Mathematics, College of
Sciences and Humanities - Kowaiyia, Shaqra University, Kingdom of
Saudi Arabia. E.mail: ismail.laraiedh@su.edu.sa}}

\maketitle
\begin{abstract}
The purpose of this paper is to introduce the notion of noncommutative BiHom-pre-Poisson
algebra. Also we establish the bimodules and matched pairs of noncommutative BiHom-(pre)-Poisson algebras and related relevant properties are also
given.
Finally, we exploit the notion of $\mathcal{O}$-operator to illustrate the relations existing between noncommutative BiHom-Poisson and noncommutative BiHom pre-Poisson algebras.\end{abstract}



\section{Introduction}
Algebraic structure appeared in the Physics literature related to string
theory, vertex models in conformal field theory, quantum mechanics and quantum
field theory, such as the q-deformed Heisenberg algebras, q-deformed oscillator algebras, q-deformed Witt, q-deformed Virasoro algebras and related q-deformations
of infinite-dimensional algebras  \cite{HnkHndSilvdcbihomfrobalg:AizawaSaito,HnkHndSilvdcbihomfrobalg:ChaiElinPop,HnkHndSilvdcbihomfrobalg:ChaiIsLukPopPresn,HnkHndSilvdcbihomfrobalg:ChaiKuLuk,HnkHndSilvdcbihomfrobalg:ChaiPopPres,HnkHndSilvdcbihomfrobalg:CurtrZachos1,HnkHndSilvdcbihomfrobalg:DamKu,HnkHndSilvdcbihomfrobalg:DaskaloyannisGendefVir,HnkHndSilvdcbihomfrobalg:Hu,HnkHndSilvdcbihomfrobalg:Kassel92,
HnkHndSilvdcbihomfrobalg:LiuKQuantumCentExt,HnkHndSilvdcbihomfrobalg:LiuKQCharQuantWittAlg,HnkHndSilvdcbihomfrobalg:LiuKQPhDthesis}.

Hom-type algebras satisfy a modified version of the Jacobi identity involving a homomorphism, and were called Hom-Lie algebras by Hartwig, Larsson and Silvestrov in \cite{hls}, \cite{larsson}. Afterwards, Hom-analogues of various classical algebraic structures have been introduced in the literature, such as Hom-associative algebras, Hom-dendriform
algebras, Hom-pre-Lie algebras Hom-(pre)-Poisson algebras etc \cite{Elkadri, AmMakh, attan0, attan1, attan2, Bak1, Makhl1, Makhl2, Makhl3, Makhl4, Yau1, Yau2, Yau3, Yau4, Yau5, Yau6}.

The notion of a noncommutative Poisson algebra was first given by Xu in \cite{Xu}.
A noncommutative Poisson algebra consists of an associative algebra together with
a Lie algebra structure, satisfying the Leibniz identity.
Noncommutative Poisson algebras are used in many fields in mathematics and physics.
Aguiar introduced the notion of a pre-Poisson algebra in \cite{A2} and constructed many examples.
A noncommutative pre-Poisson algebra contains a dendriform algebra and a pre-Lie algebra such that some compatibility conditions are satisfied. More applications of Poisson algebras, pre-Poisson algebra can be found in \cite{CohomologyPA1, CohomologyPA2, Kubo1, Kubo3, Kubo4, Van}

A generalization of this approach led the authors of \cite{GRAZIANI} to introduce BiHom-associative and BiHom-Lie
algebras. In these algebras, the associativity of the multiplication is twisted by two commuting
homomorphisms. When these two homomorphisms are equal, one recovers Hom-associative and Hom-Lie algebras.

Noncommutative BiHom-Poisson algebras was first introduced in \cite{XLi} and
studied in \cite{hadimi, Ismaiiil}. Let $(A,\cdot,\{\cdot,\cdot\},\a_1,\a_2)$ be a noncommutative BiHom-Poisson algebra. A noncommutative BiHom-Poisson $A$-module $V$ is simultaneously a BiHom-associative $A$-module $(\lambda,\b_1,\b_2,V)$ and a BiHom-Lie $A$-module $(\rho,\b_1,\b_2,V)$ satisfying the BiHom-Leibniz conditions:\\ \begin{align*}
&\rho(\alpha_1\a_2(x),\lambda(y,v))= \lambda(\{\a_2(x),y\},\b_2(v))+\lambda(\a_2(y),\rho(\alpha_1(x),v)),\\
&\rho(\mu(\a_2(x),y),\b_2(v))=\lambda(\alpha_1\a_2(x),\rho(y,v))+\lambda(\a_2(y),\rho(\alpha_1(x),v)),\end{align*} for $x,y\in A,~v\in V$, (for more details see Definition 5.3 in\cite{Ismaiiil}). A noncommutative BiHom-pre-Poisson algebra gives rise to a noncommutative BiHom-Poisson algebra
naturally through the sub-adjacent BiHom-associative algebra of the BiHom-dendriform algebra \cite{luimakhlouf} and the
sub-adjacent BiHom-Lie algebra of the BiHom-pre-Lie algebra \cite{LiuMakhloufMeninPanaite}.
We also introduce the notion of $\mathcal{O}$-operators
of noncommutative BiHom-Poisson algebra and  we will prove that given a noncommutative BiHom-Poisson algebra and an $\mathcal{O}$-operator give rise to a noncommutative BiHom-pre-Poisson algebra.
All that is illustrated by the following diagram

$$
\xymatrix{
\mbox{ BiHom-dendriform alg + BiHom-pre-Lie alg }\ar[rr] \ar[d]
                && \mbox{noncomm BiHom-pre-Poisson alg}\ar[d]\\
\ar@<-1ex>[u]\mbox{BiHom-associative alg + BiHom-Lie alg }\ar[rr]
                && \ar@<-1ex>[u]\mbox{noncomm BiHom-Poisson alg}}
$$

The paper is organized as follows. In section 2, we introduce the notions of representation and matched pair of noncommutative BiHom-Poisson algebra with a connection to a representations and matched pairs of BiHom-Lie algebra and
BiHom-associative algebra. In section 3, we establish definition of noncommutative BiHom-pre-Poisson algebra and we give some key of constructions. Their bimodule and matched pair are defined and their related relevant properties
are also given. In section 4,
we study the notion of $\mathcal{O}$-operator and we illustrate the relations existing between
noncommutative BiHom-Poisson and noncommutative BiHom pre-Poisson algebras.

Throughout this paper, all graded vector spaces are assumed to be over a field $\mathbb{K}$ of characteristic different from 2.

\section{Representation and matched pair of noncommutative BiHom-Poisson algebra}\label{sect2}
In this section we recall the definition of noncommutative BiHom-Poisson algebra \cite{XLi}
and we study the representation and the matched pair of noncommutative BiHom-Poisson algebras with a connection to a representations and matched pairs of BiHom-associative algebra and BiHom-Lie algebra.
Moreover we provide some key constructions.
\begin{defn}
A BiHom-module is a triple $(V,\alpha_V,\beta_V)$ consisting of a $\mathbb{K}$-vector space $V$ and
two linear maps $\alpha_V, \beta_V: V\longrightarrow V$  such that $\alpha_V\beta_V=\beta_V\alpha_V.$ A morphism
$f: (V,\alpha_V, \beta_V)\rightarrow (W,\alpha_W,\beta_W)$ of BiHom-modules is a linear map
 $f: V\longrightarrow W$ such that $f\alpha_V=\alpha_W f$ and
 $f\beta_V=\beta_W f.$
\end{defn}
\begin{defn}
A BiHom-algebra is a quadruple $(A,\mu,\alpha_1,\a_2)$ in which $(A,\a_1,\a_2)$ is a BiHom-module, $\mu : A^{\otimes 2} \rightarrow A$ is a bilinear map.
The BiHom-algebra $(A,\mu,\a_1,\a_2)$ is said to be  multiplicative if $\alpha_1\circ\mu=\mu\circ\alpha_{1}^{\otimes 2}$ and $\alpha_2\circ\mu=\mu\circ\alpha_{2}^{\otimes 2}$ (BiHom-multiplicativity).
\end{defn}

Let us recall now the definition and the notion of bimodule of a BiHom-associative given in \cite{GRAZIANI}.
\begin{defn}
A BiHom-associative algebra is a quadruple $(A,\mu,\a_1,\a_2)$
consisting of a vector space $A$ on which the operation $\mu: A\otimes A\rightarrow A$ and $\a_1,\a_2: A\rightarrow A$ are linear maps satisfying
\begin{eqnarray}
&&\alpha_1\circ\a_2=\a_2\circ\alpha_1,\label{aca0}\\
&&\alpha_1\circ\mu(x, y)=\mu(\alpha_1(x),\alpha_1(y)),\label{aca0.0}\\
&&\a_2\circ\mu(x, y)=\mu(\a_2(x),\a_2(y)),\label{aca0.00}\\
&&\mu(\alpha_1(x), \mu(y, z))=\mu(\mu(x, y), \a_2(z)), \label{aca}
\end{eqnarray}
for any $x, y, z \in A $.
\end{defn}
\begin{rmk}
Clearly, a Hom-associative algebra $(A,\mu,\alpha)$ can be regarded as a BiHom-associative
algebra $(A,\mu,\alpha,\alpha)$.
\end{rmk}
\begin{defn}\cite{GRAZIANI}
 Let $(A, \mu, \alpha_1, \a_2)$ be a BiHom-associative algebra. A left $A$-module is a triple $(M, \b_1 , \b_2)$, where $M$ is a linear space, $\b_1$ , $\b_2$ : $M \rightarrow M$ are linear maps, with, in addition,  another linear map: $A\otimes M \rightarrow M , a\otimes m  \mapsto a\cdot m,$ such that, for all $a, a'\in A, m \in M:$
 \begin{eqnarray*}
&&\b_1\circ\b_2= \beta_{2}\circ\b_{1},\
\b_{1}(a\cdot m)= \alpha_{1}(a)\cdot \b_{1}(m),\cr
&&\beta _{2}(a\cdot m)= \a_{2}(a)\cdot \beta_{2}(m),\
\alpha_{1}(a)\cdot(a'\cdot m)= (aa')\cdot\beta_{2}(m).
\end{eqnarray*}
\end{defn}
\begin{defn}
Let $(A, \cdot,\alpha_{1},\alpha_{2})$ be a BiHom-associative algebra, and let $(V, \beta_{1}, \beta_{2})$ be a BiHom-module. Let $ l, r: A \rightarrow gl(V),$ be two linear maps. The quintuple $(l, r, \beta_{1}, \beta_{2}, V)$ is called a bimodule of $A$ if
\begin{eqnarray}
 l(x\cdot y)\beta_{2}(v)&=& l(\alpha_{1}(x))l(y)v,\label{asss1}\\ r(x\cdot y)\beta_{1}(v)&=& r(\alpha_{2}(y))r(x)v,\\
 l(\alpha_{1}(x))r(y)v &=& r(\alpha_{2}(y))l(x)v,\\
\beta_{1}(l(x)v)&=& l(\alpha_{1}(x))\beta_{1}(v),\\ \beta_{1}(r(x)v)&=& r(\alpha_{1}(x))\beta_{1}(v),\\
\beta_{2}(l(x)v) &=& l(\alpha_{2}(x))\beta_{2}(v),\\ \beta_{2}(r(x)v)&=& r(\alpha_{2}(x))\beta_{2}(v),
\end{eqnarray}
for all $ x, y \in  A, v \in V $.
\end{defn}
\begin{prop}\label{ass1}
Let $(l, r, \beta_{1}, \beta_{2}, V)$ be a bimodule of a BiHom-associative algebra $(A, \cdot, \alpha_{1}, \alpha_{2})$. Then, the direct sum $A \oplus V$ of vector spaces is turned into a BiHom-associative algebra  by defining multiplication in $A\oplus V $ by
\begin{eqnarray*}
&&(x_{1} + v_{1}) \cdot' (x_{2} + v_{2}):=x_{1} \cdot x_{2} + (l(x_{1})v_{2} + r(x_{2})v_{1}),\cr
&&(\alpha_{1}\oplus\beta_{1})(x_{1} + v_{1}):=\alpha_{1}(x_{1}) + \beta_{1}(v_{1}),\cr &&(\alpha_{2}\oplus\beta_{2})(x_{1} + v_{1}):=\alpha_{2}(x_{1}) + \beta_{2}(v_{1}),
\end{eqnarray*}
for all $ x_{1}, x_{2} \in  A, v_{1}, v_{2} \in V$.
\end{prop}
We denote such a BiHom-associative algebra by $(A \oplus V, \cdot', \alpha_{1} + \beta_{1}, \alpha_{2} + \beta_{2}),$
or $A \times_{l, r, \alpha_{1}, \alpha_{2}, \beta_{1}, \beta_{2}} V.$
\begin{thm}\label{matched ass}\cite{double}
Let $(A,\cdot_A,\alpha_1,\alpha_2)$ and $(B,\cdot_B,\beta_1,\beta_2)$ be two BiHom-associative algebras. Suppose that there are linear maps $l_A,r_A:A\rightarrow gl(B)$
and $l_B,r_B:B\rightarrow gl(A)$ such that $(l_A,r_A,\beta_1,\beta_2,B)$ is a bimodule of $A$ and $(l_B,r_B,\alpha_1,\alpha_2,A)$ is a bimodule of $B$ satisfy
  \begin{eqnarray}\label{3}
     l_A(\alpha_1(x))(a\cdot_B b)=l_A(r_B(a)x)\beta_2(b)+(l_A(x)a)\cdot_B\beta_2(b)
                \end{eqnarray}
          \begin{eqnarray}\label{4}
     r_A(\alpha_2(x))(a\cdot_B b)=r_A(l_B(b)x)\beta_1(a)+\beta_1(a)\cdot_B(r_A(x)b)
                \end{eqnarray}
                \begin{eqnarray}\label{5}
     l_A(l_B(a)x)\beta_2(b)+(r_A(x)a)\cdot_B\beta_2(b)-r_A(r_B(b)x)\beta_1(a)-\beta_1(a)\cdot_B(l_A(x)b)=0
                \end{eqnarray}
       \begin{eqnarray}\label{6}
     l_B(\beta_1(a))(x\cdot_A y)=l_B(r_A(x)a)\alpha_2(y)+(l_B(a)x)\cdot_A\alpha_2(y)
                \end{eqnarray}
          \begin{eqnarray}\label{7}
     r_B(\beta_2(a))(x\cdot_A y)=r_B(l_A(y)a)\alpha_1(x)+\alpha_1(x)\cdot_A(r_B(a)y)
                \end{eqnarray}
                \begin{eqnarray}\label{8}
     l_B(l_A(x)a)\alpha_2(y)+(r_B(a)x)\cdot_A\alpha_2(y)-r_B(r_A(y)a)\alpha_1(x)-\alpha_1(x)\cdot_A(l_B(a)y)=0
                \end{eqnarray}
for any, $x,y\in A,~a,b\in B$. Then $(A,B,l_A,r_A,\beta_1,\beta_2,l_B,r_B,\alpha_1,\alpha_2)$ is called a matched pair of
BiHom-associative algebras. In this case, there is a BiHom-associative algebra structure on the direct sum
$A\oplus B$ of the underlying vector spaces of $A$ and $B$ given by
$$\begin{array}{llllll}
(x + a) \cdot (y + b)&:=&x \cdot_A y + (l_A(x)b + r_A(y)a)+a \cdot_B b + (l_B(a)y + r_B(b)x),\cr
(\alpha_{1}\oplus\beta_{1})(x + a)&:=&\alpha_{1}(x) + \beta_{1}(a),\cr (\alpha_{2}\oplus\beta_{2})(x + a)&:=&\alpha_{2}(x) + \beta_{2}(a).
\end{array}$$
\end{thm}
\begin{proof}
For any $x,y,z\in A$ and $a,b,c\in B$ we have
$$\begin{array}{lllllll}
&&(\a_1+\b_1)(x+a)\cdot((y+b)\cdot(z+c))\\&=&(\a_1(x)+\b_1(a))[y\cdot_A z+l_B(b)z+r_B(c)y+b\cdot c+l_A(y)c+r_A(z)b)\\
&=&\a_1(x)\cdot_A(y\cdot_A z)+\a_1(x)\cdot_A l_B(b)z+\a_1(x)\cdot_A r_B(c)y+l_B(\b_1(a))(y\cdot_A z)\\&+&l_B(\b_1(a))l_B(b)z+l_B(\b_1(a))r_B(c)y+r_B(b\cdot_B c)\a_1(x)+r_B(l_A(y)c)\a_1(x)\\&+&r_B(r_A(z)b)\a_1(x)+\b_1(a)\cdot_B(b\cdot_B c)+\b_1(a)\cdot_B l_A(y)c+\b_A(\a_1(x))l_A(y)c\\&+&l_A(\a_1(x))r_A(z)b+r_A(y\cdot_A z)\b_1(a)+r_A(l_A(b)z)\b_1(a)+r_A(r_B(c)y)\b_1(a).
\end{array}$$
In the other hand, we have
$$\begin{array}{lllllll}
&&((x+a)\cdot(y+b))\cdot(\a_2+\b_2)(z+c)\\
&=&(x\cdot_A y+l_B(a)y+r_B(b)x+a\cdot_B b+l_A(x)b+r_A(y)a)\cdot(\a_2(z)+\b_2(c))\\
&=&(x\cdot_A y)\cdot_A\a_2(z)+l_B(a)y\cdot_A\a_2(z)+r_B(b)x\cdot_A\a_2(z)+l_B(a\cdot_B b)\a_2(z)\\&+&l_B(l_A(x)b)\a_2(z)+l_B(r_A(y)a)\a_2(z)+r_B(\b_2(c))(x\cdot_A y)+r_A(\b_2(c))l_B(a)y\\&+&r_B(\b_2(c))r_B(b)x+(a\cdot_B b)\cdot_B\b_2(c)+(l_A(x)b)\cdot_B\b_2(c)+(r_A(y)a)\cdot_B\b_2(c)\\&+&r_A(\a_2(z))(a\cdot_B b)+r_A(\a_2(z))(l_A(x)b)+r_A(\a_2(z))(r_A(y)a)+l_A(x\cdot_A y)\b_2(c)\\&+&l_A(l_B(a)y)\b_2(c)+(r_B(b)x)\b_2(c).
\end{array}$$
Then by (\ref{aca}) and (\ref{3})-(\ref{8}), we deduce that $(\a_1+\b_1)(x+a)\cdot((y+b)\cdot(z+c))=((x+a)\cdot(y+b))\cdot(\a_2+\b_2)(z+c)$. This finishes the proof.
\end{proof}
We denote this BiHom-associative algebra by $A\bowtie^{l_A,r_A,\beta_1,\beta_2}_{l_B,r_B,\alpha_1,\alpha_2}B$.\\

Let us recall now the definition and the notion of bimodule of a BiHom-Lie algebra given in \cite{GRAZIANI}

\begin{defn}
A BiHom-Lie algebra is a quadruple $(A,[\cdot,\cdot],\alpha_1,\alpha_2)$ consisting
of a linear space $A$, a bilinear map
$[\cdot,\cdot]:\wedge^{2}A\rightarrow A$ and two linear maps
$\alpha_1,\alpha_2:A\rightarrow A$ satisfying $\alpha_1\circ\alpha_2=\alpha_2\circ\alpha_1$ and the following conditions, $\forall x, y, z\in A$,
\begin{enumerate}
\item
$\alpha_1([x,y])=[\alpha_1(x),\alpha_1(y)]$ and $\alpha_2([x,y])=[\alpha_2(x),\alpha_2(y)],$
\item
$[\alpha_2(x),\alpha_1(y)]=-[\alpha_2(y),\alpha_1(x)].$
\item
$[\alpha_2^{2}(x),[\alpha_2(y),\alpha_1(z)]]+[\alpha_2^{2}(z),[\alpha_2(x),\alpha_1(y)]]+[\alpha_2^{2}(y),[\alpha_2(z),\alpha_1(x)]]=0.$
\end{enumerate}
\end{defn}
\begin{defn}
Let $(A,[\cdot,\cdot],\alpha_1,\alpha_2)$ be a BiHom-Lie algebra and $(V,\beta_1,\beta_2)$ be a BiHom-module. Let $\rho:A\rightarrow gl(V)$ be a linear map. The quadruple $(\rho,\beta_1,\beta_2,,V)$ is called a representation of $A$ if for all $x,y\in A,~v\in V$, we have
\begin{eqnarray}\rho([\alpha_2(x),y])\beta_2(v)&=&\rho(\alpha_1\alpha_2(x))\circ\rho(y)v-\rho(\alpha_2(y))\circ\rho(\alpha_1(x))v,\label{repLie1}\\
\b_1(\rho(x)v)&=&\rho(\a_1(x))\b_1(v),\label{repLie2}\\
\b_2(\rho(x)v)&=&\rho(\a_2(x))\b_2(v)\label{repLie3}.
\end{eqnarray}
\end{defn}
\begin{prop}\label{pro11}
Let $(\rho, \beta_1,\beta_2, V)$ be a representation of a BiHom-Lie algebra
$(A, [\cdot,\cdot], \alpha_1,\alpha_2)$ such that $\a_1,~\b_2$ are bijectives. Then, the direct sum $A \oplus V$ of vector
spaces is turned into a BiHom-Lie algebra by defining the
multiplication in $A\oplus V $ by
$$\begin{array}{llllll}
[x_1+v_1,x_2+v_2]_\rho&=&[x_1,x_2]+\rho(x_1)v_2-\rho(\alpha_{1}^{-1}\alpha_2(x_2))\beta_1\beta_{2}^{-1}(v_1),\\
(\alpha_1\oplus\beta_1)(x_1+v_1)&=&\alpha_1(x_1)+\beta_1(v_1),\\
(\alpha_2\oplus\beta_2)(x_1+v_1)&=&\alpha_2(x_1)+\beta_2(v_1),
\end{array}$$
for all $ x_{1}, x_{2} \in  A, v_{1}, v_{2} \in V$.
\end{prop}
We denote such a BiHom-Lie algebra by $(A \oplus V, [\cdot,\cdot]_\rho, \alpha_{1} + \beta_{1}, \alpha_{2} + \beta_{2}),$
or $A \times_{\rho, \alpha_{1}, \alpha_{2}, \beta_{1}, \beta_{2}} V.$\\

Now, we introduce the notion of matched pair of BiHom-Lie algebra
\begin{thm}\label{matched Lie}
Let $(A,\{\cdot,\cdot\}_A,\alpha_1,\alpha_2)$ and $(B,\{\cdot,\cdot\}_{B},\beta_1,\beta_2)$ be two BiHom-Lie algebras. Suppose that there are linear maps $\rho_A:A\rightarrow gl(B)$
and $\rho_B:B\rightarrow gl(A)$ such that $(\rho_A,\beta_1,\beta_2,B)$ is a bimodule of $A$ and $(\rho_B,\alpha_1,\alpha_2,A)$ is a bimodule of $B$ satisfies
\begin{eqnarray}\label{Lie1}
     \rho_B(\beta_1\beta_2(a))\{\alpha_1(x),\alpha_1^{2}(y)\}_A&=&\{\rho_B(\beta_2(a))\alpha_1(x),\alpha_1^{2}\alpha_2(y)\}_A+\{\alpha_1\alpha_2(x),\rho_B(\beta_1(a))\alpha_1^{2}(y)\}_A\nonumber\\&+&\rho_B(\rho_A(\alpha_2^{2}(y))\beta_1(a))\alpha_1^{2}(x)-\rho_B(\rho_A(\alpha_2(x))\beta_1(a))\alpha_2\alpha_1^{2}(y)\nonumber\\&&
                \end{eqnarray}
             \begin{eqnarray}\label{Lie2}
     \rho_A(\alpha_1\alpha_2(x))\{\beta_1(a),\beta_1^{2}(b)\}_B&=&\{\rho_A(\alpha_2(x))\beta_1(a),\beta_1^{2}\beta_2(b)\}_B+\{\beta_1\beta_2(a),\rho_A(\alpha_1(x))\beta_1^{2}(b)\}_B\nonumber\\&+&\rho_A(\rho_B(\beta_2^{2}(b))\alpha_1(x))\beta_1^{2}(a)-\rho_A(\rho_B(\beta_2(a))\alpha_1(x))\beta_2\beta_1^{2}(b)\nonumber\\&&
                \end{eqnarray}
for any, $x,y\in A,~a,b\in B$. Then $(A,B,\rho_A,\beta_1,\beta_2,\rho_B,\alpha_1,\alpha_2)$ is called a matched pair of
BiHom-Lie algebras. Moreover, assume that $(A,\{\cdot,\cdot\}_A,\alpha_1,\alpha_2)$ and $(B,\{\cdot,\cdot\}_{B},\beta_1,\beta_2)$ be two regular BiHom-Lie algebras, then there exists a BiHom-Lie algebra structure on the vector
space $A\oplus B$ of the underlying vector spaces of $A$ and $B$ given by
$$\begin{array}{llllll}
[x+a,y+b]&:=&\{x,y\}_A+\rho_A(x)b-\rho_A(\alpha_{1}^{-1}\alpha_2(y))\beta_1\beta_{2}^{-1}(a)\cr&+&\{a,b\}_B+\rho_B(a)y-\rho_B(\beta_{1}^{-1}\beta_2(b))\alpha_1\alpha_{2}^{-1}(x),\cr
(\alpha_{1}\oplus\beta_{1})(x + a)&:=&\alpha_{1}(x) + \beta_{1}(a),\cr (\alpha_{2}\oplus\beta_{2})(x + a)&:=&\alpha_{2}(x) + \beta_{2}(a).
\end{array}$$
\end{thm}
\begin{proof}
First, we prove the BiHom-multiplicativity. For all $x,y\in A,~a,b\in B$, we have
\begin{align*}
    &(\alpha_1+\beta_1)[x+a,y+b]\\
    =&(\alpha_1+\beta_1)(\{x,y\}_A+\rho_A(x)b-\rho_A(\alpha_1^{-1}\alpha_2(y))\beta_1\beta_2^{-1}(a)\\&+\{a,b\}_B+\rho_B(a)y-\rho_B(\beta_1^{-1}\beta_2(b))\a_1\a_2^{-1}(x))\\
    =&\alpha_1(\{x,y\}_A)+\alpha_1\rho_B(a)y-\alpha_1\rho_B(\beta_1^{-1}\beta_2(b))\alpha_1\alpha_2^{-1}(x)\\
    &+\beta_1\rho_A(x)b-\b_1\rho_A(\alpha_1^{-1}\a_2(y))\b_1\b_2^{-1}(a)+\b_1\{a,b\}_B\\
    =&\{a_1(x),\a_1(y)\}_A+\rho_B(\b_1(a))\a_1(y)-\rho_B(\b_1^{-1}\b_2(\b_1(b))\a_1\a_2^{-1}(\a_1(x))\\
    &+\rho_A(\a_1(x))\b_1(b)-\rho_A(\a_1^{-1}\a_2(\a_1(y))\b_1\b_2^{-1}(\b_1(a))+\{b_1(a),\b_1(b)\}_B\\
    =&[\a_1(x)+\b_1(a),\a_1(y)+\b_1(b)]\\
    =&[(\a_1+\b_1)(x+a),(\a_1+\b_1)(y+b)].
\end{align*}
In the same way, $(\alpha_1+\beta_1)[x+a,y+b]=[(\a_1+\b_1)(x+a),\a_1+\b_1)(y+b)].$

Now, we prove the Bihom-skewsymmetry. For all $x,y\in A,~a,b\in B$, we have
\begin{align*}
    &[(\alpha_2+\beta_2)(x+a),(\alpha_1+\b_1)(y+b)]\\
    =&[\a_2(x)+\b_2(a),\a_1(y)+\b_1(b)]\\
    =&\{\a_2(x),\a_1(y)\}_A+\rho_{A}(\a_2(x))\b_1(b)-\rho_A(\a_1^{-1}\a_2(\a_1(y))\b_1\b_2^{-1}(\b_2(a))\\
    &+\{\b_2(a),\b_1(b)\}_B+\rho_B(\b_2(a))\a_1(y)-\rho_B(\b_1^{-1}\b_2(\b_1(b)))\a_1\a_2^{-1}(\a_2(x))\\
    =&-\{a_2(y),\a_1(x)\}_A-\rho_A(\a_2(y))\b_1(a)+\rho_A(\a_1^{-1}\a_2(\a_1(x))\b_1\b_2^{-1}(\b_2(b))-\{\b_2(b),\b_1(a)\}_B\\&-\rho_B(\b_2(b))\a_1(x)+\rho_B(\b_1^{-1}\b_2(\b_1(a)))\a_1\a_2^{-1}(\a_2(y))\\
    =&-[(\a_2+\b_2)(y+b),(\a_1+\b_1)(x+a)].
\end{align*}
Finally, we prove the Bihom-Jacobi identity. For any $x,y,z\in A,~a,b,c\in B$, we have
\begin{align*}
&[(\a_2+\b_2)^{2}(x+a),[(\a_2+\b_2)(y+b),(\a_1+\b_1)(z+c)]]\\
=&[\a_2^{2}(x)+\b_2^{2}(a),[\a_2(y)+\b_2(b),\a_1(z)+\b_1(c)]]\\
=&[\a_2^{2}(x)+\b_2^{2}(a),\{\a_2(y),\a_1(z)\}_1+\rho_A(\a_2(y))\b_1(c)-\rho_A(\a_1^{-1}\a_2(\a_1(z))))\b_1\b_2^{-1}(\b_2(b))\\&+\{\b_2(b),\b_1(c)\}_B+\rho_B(\b_2(b))\a_1(z)-\rho_B(\b_1^{-1}\b_2(\b_1(c)))\a_1\a_2^{-1}(\a_2(y))]\\
=&[\a_2^{2}(x)+\b_2^{2}(a),\{\a_2(y),\a_1(z)\}_1+\rho_A(\a_2(y))\b_1(c)-\rho_A(\a_2(z))\b_1(b)\\&+\{\b_2(b),\b_1(c)\}_B+\rho_B(\b_2(b))\a_1(z)-\rho_B(\b_2(c))\a_1(y)]\\
=&\{\a_2^{2}(x),\{\a_2(y),\a_1(z)\}_A\}_A+\{\a_2^{2}(x),\rho_B(\b_2(b))\a_1(z)\}_A-\{\a_2^{2}(x),\rho_B(\b_2(c))\a_1(y)\}_A\\&+\rho_A(\a_2^{2}(x))\rho_A(\a_2(y))\b_1(c)-\rho_A(\a_2^{2}(x))\rho_A(\a_2(z))\b_1(b)+\rho_A(\a_2^{2}(x))\{\b_2(b),\b_1(c)\}_B\\&-\rho_A(\{\a_1^{-1}\a_2^{2}(y),\a_2(z)\}_A)\b_1\b_2(a)-\rho_A(\rho_B(\b_1^{-1}\b_2^{2}(b))\a_2(z))\b_1\b_2(a)\\&+\rho_A(\rho_B(\b_1^{-1}\b_2^{2}(c))\a_2(y))\b_1\b_2(a)+\{\b_2^{2}(a),\rho_A(\a_2(y))\b_1(c)\}_B-\{\b_2^{2}(a),\rho_A(\a_2(z))\b_1(b)\}_B\\&+\{\b_2^{2}(a),\{\b_2(b),\b_1(c)\}_B\}_B+\rho_B(\b_2^{2}(a)\{\a_2(y),\a_1(z)\}_A+\rho_B(\b_2^{2}(a))\rho_B(\b_2(b))\a_1(z)\\&-\rho_B(\b_2^{2}(a))\rho_B(\b_2(c))\a_1(y)-\rho_B(\rho_A(\a_1^{-1}\a_2^{2}(y))\b_2(c))\a_1\a_2(x)\\&+\rho_B(\rho_A(\a_1^{-1}\a_2^{2}(z))\b_1(b))\a_1\a_2(x)-\rho_B(\{\b_1^{-1}\b_2^{2}(b),\b_2(c)\}_B)\a_1\a_2(x).
\end{align*}
By a direct computation we verify that $\circlearrowleft_{x,y,z}[(\a_2+\b_2)^{2}(x+a),[(\a_2+\b_2)(y+b),(\a_1+\b_1)(z+c)]]=0$. This ends the proof.
\end{proof}
We denote this BiHom-Lie algebra by $A\bowtie^{\rho_A,\beta_1,\beta_2}_{\rho_B,\alpha_1,\alpha_2}B$.
\begin{defn}
A noncommutative BiHom-Poisson algebra is a $5$-tuple
$(A,\cdot,\{\cdot,\cdot\},\alpha_1,\alpha_2)$, where $(A,\cdot,\alpha_1,\alpha_2)$ is a
BiHom-associative algebra and
$(A,\{\cdot,\cdot\},\alpha_1,\alpha_2)$ is a BiHom-Lie algebra, such that the
BiHom-Leibniz identity:
\begin{equation}\label{leibniz}\{\alpha_1\alpha_2(x),y\cdot
z\}=\{\alpha_2(x),y\}\cdot\alpha_2(z)+\alpha_2(y)\cdot\{\alpha_1(x),z\}.\end{equation}
\end{defn}
\begin{prop}\label{adimi}\cite{hadimi}
Let $(A, \cdot, \alpha_1,\a_2)$ be a regular BiHom-associative algebra. Then\\
$A^{-} = (A, \{\cdot,\cdot \}, \cdot, \alpha,\beta)$
is a regular noncommutative BiHom-Poisson algebra, where for all $x,y\in A$, $\{x,y\}=x\cdot y-\alpha_1^{-1}\a_2(y)\cdot\alpha_1\a_2^{-1}(x)$.
\end{prop}
In the following we introduce the notions of representation and matched pair of noncommutative BiHom-Poisson algebras.
\begin{defn}
Let $(A,\cdot,\{\cdot,\cdot\},\alpha_1,\alpha_2)$ be a noncommutative BiHom-Poisson algebra. A representation of $A$ is a $6$-tuple $(l,r,\rho,\beta_1,\beta_2,V)$ such that $(l,r,\beta_1,\beta_2,V)$ is a bimodule of the BiHom-associative algebra $(A,\cdot,\alpha_1,\alpha_2)$ and $(\rho,\beta_1,\beta_2,V)$ is a representation of the BiHom-Lie algebra $(A,\{\cdot,\cdot\},\alpha_1,\alpha_2)$ satisfying, for all $x,y\in A,~v\in V.$
\begin{eqnarray}
l(\{\alpha_2(x),y\})\beta_2(v)&=&\rho(\alpha_1\alpha_2(x))l(y)v-l(\alpha_2(y))\rho(\alpha_1(x))v\label{isma1.1}\\
r(\{\alpha_1(x),y\})\beta_2(v)&=&\rho(\alpha_1\alpha_2(x))r(y)v-r(\alpha_2(y))\rho(\alpha_2(x))v\label{isma1.2}\\
\rho(x\cdot y)\beta_1\beta_2(v)&=&l(\alpha_1(x))\rho(y)\beta_1(v)+r(\alpha_1(y))\rho(x)\beta_2(v)\label{isma1.3}.
\end{eqnarray}
\end{defn}
\begin{prop}\label{pro2}
Let $(l,r,\rho,\beta_1,\beta_2,V)$ be a representation of noncommutative BiHom-Poisson algebra $(A,\cdot,\{\cdot,\cdot\},\alpha_1,\alpha_2)$ such that $\a_1,~\b_2$ are bijectives. Then $(A\oplus V,\cdot',\{\cdot,\cdot\}',\alpha_1+\beta_1,\alpha_2+\beta_2)$ is a noncommutative BiHom-Poisson algebra, where $(A\oplus V,\cdot',\alpha_1+\beta_1,\alpha_2+\beta_2)$ is the semi-direct product BiHom-associative algebra $A \times_{l,r, \alpha_{1}, \alpha_{2}, \beta_{1}, \beta_{2}} V$ and $(A\oplus V,\{\cdot,\cdot\}',\alpha_1+\beta_1,\alpha_2+\beta_2)$ is the semi-direct product BiHom-Lie algebra $A \times_{\rho, \alpha_{1}, \alpha_{2}, \beta_{1}, \beta_{2}} V$\end{prop}
\begin{proof}
We prove only the BiHom-Leibniz identity. For all $x_1,x_2,x_3\in A,~v_1,v_2,v_3\in V.$
\begin{align*}
&\{(\a_1\a_2+\b_1\b_2)(x_1+v_1),(x_2+v_2)\cdot'(x_3+v_3)\}'\\&-\{(\a_2+\b_2)(x_2+v_2)\cdot'\{(\a_1+\b_1)(x_1+v_1),x_3+v_3\}'\\
&-(\a_2+\b_2)(x_2+v_2)\cdot'\{(\a_1+\b_1)(x_1+v_1),x_3+v_3\}'\\
=&\{\a_1\a_2(x_1)+\b_1\b_2(v_1),x_2\cdot x_3+l(x_2)v_3+r(x_3)v_2\}'\\&-\Big(\{\a_2(x_1),x_2\}+\rho(\a_2(x_1))v_2-\rho(\a_1^{-1}\a_2(x_2))\b_1(v_1)\Big)\cdot'(\a_2(x_1)+\b_2(v_3))\\&-(\a_2(x_2)+\b_2(v_2))\cdot'\Big(\{\a_1(x_1), x_3\}+\rho(\a_1(x_1))v_3-\rho(\a_1^{-1}\a_2(x_3))\b_1^{2}\b_2^{-1}(v_1)\Big)\\
=&\{\a_1\a_2(x_1),x_2\cdot x_3\}+\rho(\a_1\a_2(x_2))l(x_2)v_3\\&+\rho(\a_1\a_2(x_1))r(x_3) v_2-\rho(\a_1^{-1}\a_2(x_2(x_3))\b_1^{2}(v_1)\\&-\{\a_2(x_1),x_2\}\cdot\a_2(x_3)-l(\{a_2(x_1),x_2\})\b_2(v_3))\\&-r(\a_2(x_3))\rho(\a_2(x_1))v_2+r(\a_2(x_3))\rho(\a_1^{-1}\a_2(x_2))\b_1(v_1)\\&-\a_2(x_2)\cdot\{\a_1(x_1),x_3\}-l(\a_2(x_2))\rho(\a_1(x_1))v_3\\&+l(\a_2(x_2))\rho(\a_1^{-1}\a_2(x_3))\b_1^{2}\b_2^{_1}(v_1)-r(\{\a_1(x_1),x_3\})\b_2\\=&\Big(\{\a_1\a_2(x_1),x_2\cdot x_3\}-\{`\a_2(x_1),x_2\}\cdot\a_2(x_3)-\a_2(x_2)\{\a_1(x_1),x_3\}\Big)\\
&+\Big(\rho(\a_1\a_2(x_1))l(x_2)v_3-l(\{\a_2(x_1),x_2\})\b_2(v_3)-l(\a_2(x_2))\rho(\a_1(x_1))v_3\Big)\\&+\Big(\rho(\a_1\a_2(x_1))r(x_3)v_2-r(\a_2(x_3))\rho(\a_2(x_1))v_2-r(\{\a_1(x_1)),x_3\})\b_2(v_2)\Big)\\&-\Big(\rho(\a_1^{-1}\a_2(x_2\cdot x_3))\b_1^{2}(v_1)+r(\a_2(x_3))\rho(\a_1^{-1}\a_2(x_2))\b_1(v_1)\\&+l(\a_2(x_2))\rho(\a_1^{-1}\a_2(x_3))\b_1^{2}\b_{2}^{-1}(v_1)\Big)
=0+0+0+0=0.
\end{align*}
Then $(A\oplus V,\cdot',\{\cdot,\cdot\}',\alpha_1+\beta_1,\alpha_2+\beta_2)$ is a noncommutative BiHom-Poisson algebra and we denote by $A \times_{l,r,\rho, \alpha_{1}, \alpha_{2}, \beta_{1}, \beta_{2}} V.$
\end{proof}
\begin{ex}
Let $(A,\cdot,\{\cdot,\cdot\},\alpha_1,\a_2)$ be a noncommutative BiHom-Poisson algebra.
Then $(L_{\cdot},R_{\cdot},ad,\alpha_1,\a_2,A)$ is a regular
representation of $A$, where $L_{\cdot}(a)b=a\cdot b,~R_{\cdot}(a)b=b\cdot
a$ and
$ad(a)b=[a,b]$, for all $a,b\in A$.
\end{ex}
\begin{thm}
Let $(A,\cdot_A,\{\cdot,\cdot\}_A,\alpha_1,\alpha_2)$ and $(B,\cdot_B,\{\cdot,\cdot\}_{B},\beta_1,\beta_2)$ be two noncommutative BiHom-Poisson algebras. Suppose that there are linear maps $l_A,r_A,\rho_A:A\rightarrow gl(B)$
and $l_B,r_B,\rho_B:B\rightarrow gl(A)$ such that $A\bowtie^{\rho_A,\beta_1,\beta_2}_{\rho_B,\alpha_1,\alpha_2}B$ is a matched pair of BiHom-Lie algebras and  $A\bowtie^{l_A,r_A,\beta_1,\beta_2}_{l_B,r_B,\alpha_1,\alpha_2}B$ is a matched pair of BiHom-associative algebras and for all $x,y\in A,~a,b\in B$, the following equalities hold:
     \begin{eqnarray}\label{101}
\rho_A(\alpha_1\alpha_2^{2}(x))(\beta_1(a)\cdot_B \beta_1(b))&=&(\rho_A(\alpha_2^{2}(x))\beta_1(a))\cdot_B \beta_1\beta_2(b)+\beta_1\beta_2(a)\cdot_B\rho_A(\alpha_1\alpha_2(x))\beta_1(b)\nonumber\\&-&l_A(\rho_B(\beta_2(a))\alpha_2(x))\beta_1\beta_2(b)-r_A(\rho_B(\beta_2(b))\alpha_1^{2}(x))\beta_1\beta_2(a),\nonumber\\&&
\end{eqnarray}
\begin{eqnarray}\label{102}
l_A(\alpha_1\alpha_2(x))\{\beta_1(a),\beta_1(b)\}_B&=&\{\beta_1\beta_2(a),l_A(\alpha_1(x))\beta_1(b)\}_B-\rho_A(r_B(\beta_2(b))\alpha_2^{2}(x))\beta_1^{2}(a)\nonumber\\&-&l_A(\rho_B(\beta_2(a))\alpha_1(x))\beta_1\beta_2(b)+(\rho_A(\alpha_2(x))\beta_1(a))\cdot_B\beta_1\beta_2(b),\nonumber\\&&
\end{eqnarray}
 \begin{eqnarray}\label{103}
\rho_B(\beta_1\beta_2^{2}(a))(\alpha_1(x)\cdot_A \alpha_1(y))&=&(\rho_B(\beta_2^{2}(a))\alpha_1(x))\cdot_A \alpha_1\alpha_2(y)+\alpha_1\alpha_2(x)\cdot_A\rho_B(\beta_1\beta_2(a))\alpha_1(y)\nonumber\\&-&l_B(\rho_A(\alpha_2(x))\beta_2(a))\alpha_1\alpha_2(y)-r_B(\rho_A(\alpha_2(y))\beta_1^{2}(a))\alpha_1\alpha_2(x),\nonumber\\&&
\end{eqnarray}
\begin{eqnarray}\label{104}
l_B(\beta_1\beta_2(a))\{\alpha_1(x),\alpha_1(y)\}_A&=&\{\alpha_1\alpha_2(x),l_B(\beta_1(a))\alpha_1(y)\}_A-\rho_B(r_A(\alpha_2(y))\beta_2^{2}(a))\alpha_1^{2}(x)\nonumber\\&-&l_B(\rho_A(\beta_2(x))\beta_1(a))\alpha_1\alpha_2(y)+(\rho_B(\beta_2(a))\alpha_1(x))\cdot_A\alpha_1\alpha_2(y).\nonumber\\&&
\end{eqnarray}
Then $(A,B,l_A,r_A,\rho_A,\beta_1,\beta_2,l_B,r_B,\rho_B,\alpha_1,\alpha_2)$ is called a matched pair of noncommutative
BiHom-Poisson algebras. Moreover, assume that $(A,\{\cdot,\cdot\}_A,\alpha_1,\alpha_2)$ and $(B,\{\cdot,\cdot\}_{B},\beta_1,\beta_2)$ be two regular noncommutative BiHom-Poisson algebras, then, there exists a noncommutative BiHom-Poisson algebra structure on the direct sum
$A\oplus B$ of the underlying vector spaces of $A$ and $B$ given by
$$\begin{array}{llllll}
(x + a) \cdot (y + b)&:=&x \cdot_A y + (l_A(x)b + r_A(y)a)+a \cdot_B b + (l_B(a)y + r_B(b)x),\cr
[x+a,y+b]&:=&\{x,y\}_A+\rho_A(x)b-\rho_A(\alpha_{1}^{-1}\alpha_2(y))\beta_1\beta_{2}^{-1}(a)\cr&+&\{a,b\}_B+\rho_B(a)y-\rho_B(\beta_{1}^{-1}\beta_2(b))\alpha_1\alpha_{2}^{-1}(x),\cr
(\alpha_{1}\oplus\beta_{1})(x + a)&:=&\alpha_{1}(x) + \beta_{1}(a),\cr (\alpha_{2}\oplus\beta_{2})(x + a)&:=&\alpha_{2}(x) + \beta_{2}(a).
\end{array}$$
for any $x,y\in A,~a,b\in B.$
\end{thm}
\begin{proof}
By Theorem \ref{matched ass} and Theorem \ref{matched Lie}, we deduce that $(A\oplus B,\cdot,\a_1+\b_1,\a_2+\b_2)$
is a BiHom-associative algebra and $(A\oplus B,[\cdot,\cdot],\a_1+\b_1,\a_2+\b_2)$ is a BiHom-Lie algebra.
Now, the
rest, it is easy ( in a similar way as for Proposition \ref{matched Lie} and Proposition \ref{matched ass}) to verify the BiHom-Leibniz identity satisfied.
\end{proof}
We denote this noncommutative BiHom-Poisson algebra by $A\bowtie^{l_A,r_A,\rho_A,\beta_1,\beta_2}_{l_B,r_B,\rho_B,\alpha_1,\alpha_2}B$.
\section{Bimodules and matched pairs of noncommutative BiHom-pre-Poisson algebras}
In this section we introduce the definition and bimodule of a noncommutative BiHom-pre-Poisson algebra. We also
establish the matched pair of noncommutative BiHom-pre-Poisson algebra and equivalently link them to a matched pair of
their underlying noncommutative BiHom-Poisson algebras.
\begin{defn}\cite{luimakhlouf}
A BiHom-pre-Lie algebra $(A,\ast, \alpha_1,\a_2)$ is a vector space $A$
equipped with a bilinear product $\ast:A\otimes A\rightarrow A,$ and two linear maps
$\alpha_1,\a_2\in End(A)$, such that for all $x, y, z\in A,~\alpha_1(x\ast y)
=\alpha_1(x)\ast\alpha_1(y),~\a_2(x\ast y)
=\a_2(x)\ast\a_2(y)$ and the following equality is satisfied:
\begin{equation}\label{klkl}
(\a_2(x)\ast \a_1(y))\ast\a_2(z)-\a_1\a_2(x)\ast(\a_1(y)\ast z)=(\a_2(y)\ast
\a_1(x))\ast\a_2(z)-\a_1\a_2(y)\ast(\a_1(x)\ast z).
\end{equation}
The equation (\ref{klkl}) is called BiHom-pre-Lie identity.
\end{defn}
\begin{lem}\label{lem1}\cite{luimakhlouf}
Let $(A,\ast,\a_1,\a_2)$ be a regular BiHom-pre-Lie algebra. Then
$(A,[\cdot,\cdot],\a_1,\a_2)$ is a BiHom-Lie algebra with
$$[x, y]=x\ast y-\a_1^{-1}\a_2(y)\ast \a_1\a_2^{-1}(x),$$
for any $x,y \in A$. We say that $(A,[\cdot,\cdot],\a_1,\a_2)$ is the
sub-adjacent BiHom-Lie algebra of $(A,\ast,\a_1,\a_2)$ and denoted by $A^{c}$.
\end{lem}
Let us recall now the notion of bimodule of a BiHom-pre-Lie algebra given in \cite{left symm}.
\begin{defn}
Let $(A, \ast, \alpha_{1}, \alpha_{2})$ be a BiHom-pre-Lie algebra, and let $(V, \beta_{1}, \beta_{2})$ be a BiHom-module. Let $ l_\ast, r_\ast: A \rightarrow gl(V) $ be two linear maps. The quintuple $(l_\ast, r_\ast, \beta_{1}, \beta_{2}, V)$ is called a bimodule of $A$ if for all $ x, y \in  A, v \in V $
\begin{eqnarray}
l_\ast(\{\alpha_{2}(x),\alpha_{1}(y)\})\beta_{2}(v)\nonumber&=&
l_\ast(\alpha_{1}\alpha_{2}(x))l_\ast(\alpha_{1}(y))v\nonumber\\
&-&l_\ast(\alpha_{1}\alpha_{2}(y))l_\ast(\alpha_{1}(x))v\\
r_\ast(\a_2(y))\rho(\a_2(x))\b_1(v)&=&l_\ast(\a_1\a_2(x))r_\ast(y)\b_1(v)\nonumber\\&-&r_\ast(\a_1(x)\ast y)\b_1\b_2(v)\\
\beta_{1}(l_\ast(x)v)&=& l_\ast(\alpha_{1}(x))\beta_{1}(v),\\ \beta_{1}(r_\ast(x)v)&=& r_\ast(\alpha_{1}(x))\beta_{1}(v),\\
\beta_{2}(l_\ast(x)v) &=& l_\ast(\alpha_{2}(x))\beta_{2}(v),\\\beta_{2}(r_\ast(x)v)&=& r_\ast(\alpha_{2}(x))\beta_{2}(v)
\end{eqnarray}
.where $\{\a_2(x),\a_1(y)\}=\a_2(x)\ast\a_1(y)-\a_2(y)\ast\a_1(x)$ and $(\rho\circ\a_2)\b_1=(l_\ast\circ\a_2)\b_1-(r_\ast\circ\a_1)\b_2.$
\end{defn}
\begin{prop}
Let $(l_\ast, r_\ast, \beta_{1}, \beta_{2}, V)$ be a bimodule of a BiHom-pre-Lie algebra
$(A, \ast, \alpha_{1}, \alpha_{2})$. Then, the direct sum $A \oplus V$ of vector spaces is turned into a
 BiHom-pre-Lie algebra  by defining multiplication in $A \oplus V $ by
\begin{eqnarray*}
&&(x_{1} + v_{1}) \ast' (x_{2} + v_{2}):=x_{1} \ast x_{2} + (l_\ast(x_{1})v_{2} + r_\ast(x_{2})v_{1}),\cr
&&(\alpha_{1}\oplus\beta_{1})(x_{1} + v_{1}):=\alpha_{1}(x_{1}) + \beta_{1}(v_{1}),\cr&&(\alpha_{2}\oplus\beta_{2})(x_{1} + v_{1}):=\alpha_{2}(x_{1}) + \beta_{2}(v_{1}),
\end{eqnarray*}
for all $ x_{1}, x_{2} \in  A, v_{1}, v_{2} \in V$.
\end{prop}
We denote such a BiHom-pre-Lie algebra by $(A\oplus V, \ast', \alpha_{1} + \beta_{1}, \alpha_{2} + \beta_{2}),$\\
or $A\times_{l_\ast, r_\ast, \alpha_{1}, \alpha_{2}, \beta_{1}, \beta_{2}} V.$\\
\begin{prop}\label{propa1}
Let ($ l_{\ast},
r_{\ast}, \beta_1,\b_2, V$) be a bimodule of a regular BiHom-pre-Lie algebra
$(A,\ast, \alpha_1,\a_2)$ such that $\b_1$ is bijective. Let $(A, \{\cdot,\cdot\}, \alpha_1,\a_2)$ be the
subadjacent BiHom-Lie algebra of $(A, \ast, \alpha_1,\a_2)$. Then ($l_\ast-(r_\ast\circ\a_1\a_2^{-1})\b_1^{-1}\b_2,\b_1,\b_2,V$) is a
representation of Lie algebra $(A,\{\cdot,\cdot\}, \alpha_1,\a_2)$.
\end{prop}\begin{proof}
For all $x,y\in A$
\begin{align*}
&\b_1\circ (l_\ast(x)-r_\ast(\a_1\a_2^{-1}(x))\b_1^{-1}\b_2)(v)\\
=&\b_1\circ l_\ast(x)v-\b_1\circ(r_\ast(\a_1\a_2^{-1}(x))\b_1^{-1}\b_2(v)\\
=&l_\ast(\a_1(x))\b_1(v)-(r_\ast\circ\a_1\a_2^{-1}(\a_1(x)))\b_1^{-1}\b_2(\b_1(v))\\
=&(l_\ast-(r_\ast\circ\a_1\a_2^{-1})\b_1^{-1}\b_2)(\a_1(x))\b_1(v).
\end{align*}
In the same way, we have $\b_2\circ (l_\ast-(r_\ast\circ\a_1\a_2^{-1}))(x)\b_1^{-1}\b_2)(v)=(l_\ast-(r_\ast\circ\a_1\a_2^{-1})\b_1^{-1}\b_2)(\a_2(x))\b_2(v)$.

Finally, for all $x,y\in A$, we have
\begin{align*}
&(l_\ast-(r_\ast\circ\a_1\a_2^{-1})\b_1^{-1}\b_2)(\a_1\a_2(x))\circ(l_\ast-(r_\ast\circ\a_1\a_2^{-1})\b_1^{-1}\b_2)(y)v\\
&-(l_\ast-(r_\ast\circ\a_1\a_2^{-1})\b_1^{-1}\b_2)(\a_2(y)\circ(l_\ast-(r_\ast\circ\a_1\a_2^{-1})\b_1^{-1}\b_2)(\a_1(x))v\\
=&l_\ast(\a_1\a_2(x))\circ l_\ast(y)v-(r_\ast\circ\a_1^{2}(x))\b_1^{-1}\b_2\circ l_\ast(y)v\\
&-l_\ast(\a_1\a_2(x))\circ(r_\ast\circ\a_1\a_2^{-1}(y))\b_1^{-1}\b_2(v)+(r_\ast\circ\a_1^{2}(x))\b_1^{-1}\b_2\circ(r_\ast\circ\a_1\a_2^{-1}(y))\b_1^{-1}\b_2(v)\\
&-l_\ast(\a_2(y))\circ l_\ast(\a_1(x))v+l_\ast(\a_2(y))(r_\ast\circ\a_1^{2}\a_2^{-1}(x))\b_1^{-1}\b_2(v)\\
&+(r_\ast\circ\a_1(y))\b_1^{-1}\b_2\circ l_\ast(\a_1(x))v-(r_\ast\circ\a_1(y))\b_1^{-1}\b_2\circ(r_\ast\circ\a_1^{2}\a_2^{-1}(x))\b_1^{-1}\b_2(v)\\
=&l_\ast(\a_1\a_2(x))\circ l_\ast(y)v-(r_\ast\circ\a_1^{2}(x))\circ l_\ast(\a_1^{-1}\a_2(y))\b_1^{-1}\b_2(v)\\
&-l_\ast(\a_1\a_2(x))\circ(r_\ast\circ\a_1\a_2^{-1}(y))\b_1^{-1}\b_2(v)+(r_\ast\circ\a_1^{2}(x))\circ(r_\ast(y))\b_1^{-2}\b_2^{2}(v)\\
&-l_\ast(\a_2(y))\circ l_\ast(\a_1(x))v+l_\ast(\a_2(y))(r_\ast\circ\a_1^{2}\a_2^{-1}(x))\b_1^{-1}\b_2(v)\\
&+(r_\ast\circ\a_1(y))\circ l_\ast(\a_2(x))\b_1^{-1}\b_2(v)-(r_\ast\circ\a_1(y))\circ(r_\ast\circ\a_1(x))\b_1^{-2}\b_2^{2}(v)\\
=&\Big(l_\ast(\a_1\a_2(x))\circ l_\ast(y)v-l_\ast(\a_2(y))\circ l_\ast(\a_1(x))v\Big)\\
&+\Big(-l_\ast(\a_1\a_2(x))\circ(r_\ast\circ\a_1\a_2^{-1}(y))\b_1^{-1}\b_2(v)+(r_\ast\circ\a_1(y))\circ l_\ast(\a_2(x))\b_1^{-1}\b_2(v)\\&-(r_\ast\circ\a_1(y))\circ(r_\ast\circ\a_1(x))\b_1^{-2}\b_2^{2}(v)\Big)
-\Big( (r_\ast\circ\a_1^{2}(x))\circ l_\ast(\a_1^{-1}\a_2(y))\b_1^{-1}\b_2(v)\\&-(r_\ast\circ\a_1^{2}(x))\circ(r_\ast(y))\b_1^{-2}\b_2^{2}(v)-l_\ast(\a_2(y))(r_\ast\circ\a_1^{2}\a_2^{-1}(x))\b_1^{-1}\b_2(v)\Big)\\
=&l_\ast(\{\a_2(x),y\})\b_2(v)-r_\ast(\a_1(x)\ast\a_1\a_2^{-1}(y))\b_1^{-1}\b_2^{2}(v)+r_\ast(y\ast\a_1^{2}\a_2^{-1}(x))\b_1^{-1}\b_2^{2}(v)\\
=&l_\ast(\{\a_2(x),y\})\b_2(v)-r_\ast(\a_1(x)\ast\a_1\a_2^{-1}(y)-y\ast\a_1^{2}\a_2^{-1}(x))\b_1^{-1}\b_2^{2}(v)\\
=&l_\ast(\{\a_2(x),y\})\b_2(v)-r_\ast(\{\a_1(x),\a_1\a_2^{-1}(y)\})\b_1^{-1}\b_2^{2}(v)\\
=&\rho(\{\a_2(x),y\})\b_2(v).
\end{align*}
Therefore, (\ref{repLie1})-(\ref{repLie3}) are satisfied.
\end{proof}

Now, we introduce the notion of matched pair of BiHom-pre-Lie algebra:
\begin{thm}
Let $(A,\ast_A,\alpha_1,\alpha_2)$ and $(B,\ast_{B},\beta_1,\beta_2)$ be two  BiHom-pre-Lie algebras. Suppose that there are linear maps $l_{\ast_A},r_{\ast_A}:A\rightarrow gl(B)$
and $l_{\ast_B},r_{\ast_B}:B\rightarrow gl(A)$ such that $(l_{\ast_A},r_{\ast_A},\beta_1,\beta_2,B)$ is a bimodule of $A$ and $(l_{\ast_B},r_{\ast_B},\alpha_1,\alpha_2,A)$ is a bimodule of $B$ satisfy for any, $x,y\in A,~a,b\in B$:
\begin{eqnarray}\label{Lie11}
    &&r_{\ast_A}(\a_2(x))\{\b_2(a),\b_1(b)\}_B=r_{\ast_A}(l_{\ast_B}(\b_1(b))x)\b_1\b_2(a)\nonumber\\&&-r_{\ast_A}(l_{\ast_B}(\b_1(a)x)\b_1\b_2(b)+\b_1\b_2(a)\ast_Br_{\ast_A}(x)\b_1(b)\nonumber\\&&-\b_1\b_2(b)\ast_B r_{\ast_A}(x)\b_1(a),\label{Lie11}\\
    &&l_{\ast_A}(\a_1\a_2(x))(\b_1(a)\ast_B b)=(\rho_A(\a_2(x))\b_1(a))\ast_B\b_{2}(b)\nonumber\\&&-l_{\ast_A}(\rho_B(\b_2(a))\a_1(x))\b_2(b)+\b_1\b_2(a)\ast_B(l_{\ast_A}(\a_1(x))b)\nonumber\\&&+r_{\ast_A}(r_{\ast_B}(b)\a_1(x))\b_1\b_2(a),\label{Lie12}\\
      &&r_{\ast_B}(\b_2(a))\{\a_2(x),\a_1(y)\}_A=r_{\ast_B}(l_{\ast_A}(\a_1(y))a)\a_1\a_2(x)\nonumber\\&&-r_{\ast_B}(l_{\ast_A}(\a_1(x)a)\a_1\a_2(y)+\a_1\a_2(x)\ast_A r_{\ast_B}(a)\a_1(y)\nonumber\\&&-\a_1\a_2(y)\ast_A r_{\ast_B}(a)\a_1(x),\label{Lie13}\\
    &&l_{\ast_B}(\b_1\b_2(a))(\a_1(x)\ast_A y)=(\rho_B(\b_2(a))\a_1(x))\ast_A\a_{2}(y)\nonumber\\&&-l_{\ast_B}(\rho_A(\a_2(x))\b_1(a))\a_2(y)+\a_1\a_2(x)\ast_A(l_{\ast_B}(\b_1(a))y)\nonumber\\&&+r_{\ast_B}(r_{\ast_A}(y)\b_1(a))\a_1\a_2(x),\label{Lie14}
                \end{eqnarray}
where
\begin{align*}\{\a_2(x),\a_1(y)\}_A&=&\a_2(x)\ast_A\a_1(y)-\a_2(y)\ast_A\a_1(x),~~(\rho_A\circ\a_2)\b_1&=&(l_{\ast_A}\circ\a_2)\b_1-(r_{\ast_A}\circ\a_1)\b_2,\\
\{\b_2(a),\b_1(b)\}_B&=&\b_2(a)\ast_B\b_1(b)-\b_2(b)\ast_B\b_1(a),~~(\rho_B\circ\b_2)\a_1&=&(l_{\ast_B}\circ\b_2)\a_1-(r_{\ast_B}\circ\b_1)\a_2.
\end{align*}
Then $(A,B,l_{\ast_A},r_{\ast_A},\beta_1,\beta_2,l_{\ast_B},r_{\ast_B},\alpha_1,\alpha_2)$ is called a matched pair of
BiHom-pre-Lie algebras. In this case, there exists a BiHom-pre-Lie algebra structure on the vector
space $A\oplus B$ of the underlying vector spaces of $A$ and $B$ given by
$$\begin{array}{llllll}
(x + a) \ast (y + b)&:=&x \ast_A y + (l_{\ast_A}(x)b + r_{\ast_A}(y)a)+a \ast_B b + (l_{\ast_B}(a)y + r_{\ast_B}(b)x),\cr
(\alpha_{1}\oplus\beta_{1})(x + a)&:=&\alpha_{1}(x) + \beta_{1}(a),\cr (\alpha_{2}\oplus\beta_{2})(x + a)&:=&\alpha_{2}(x) + \beta_{2}(a).
\end{array}$$
\end{thm}
\begin{proof}
The proof is obtained in a similar way as for Theorem \ref{matched ass}.
\end{proof}
We denote this BiHom-pre-Lie algebra by $A\bowtie^{l_{\ast_A},r_{\ast_A},\beta_1,\beta_2}_{l_{\ast_B},r_{\ast_B},\alpha_1,\alpha_2}B$.\\
\begin{prop}\label{Matchedd1}
Let $(A, B, l_{\ast_{A}}, r_{\ast_{A}}, \beta_1,\b_2,
l_{\ast_{B}}, r_{\ast_{B}}, \alpha_1,\a_2) $ be a matched pair of regular BiHom-pre-Lie algebras $(A,\ast_A, \alpha_1,\a_2)$ and $(B, \ast_B,\beta_1,\b_2)$.
Then, $(A, B, l_{\ast_{A}}-(r_{\ast_{A}}\circ\a_1\a_2^{-1})\b_1^{-1}\b_2,\\\b_1,\b_2 ,l_{\ast_{B}}-(r_{\ast_{B}}\circ\b_1\b_2^{-1})\a_1^{-1}\a_2,\a_1,\a_2)$ is a matched pair of the associated
BiHom-Lie algebras $(A,\{\cdot,\cdot\}_A, \alpha_1,\a_2)$ and $(B,\{\cdot,\cdot\}_B,\beta_1,\b_2).$
\end{prop}
\begin{proof}
Let $(A, B, l_{\ast_{A}}, r_{\ast_{A}}, \beta_1,\b_2,
l_{\ast_{B}}, r_{\ast_{B}}, \alpha_1,\a_2) $ be a matched pair of regular BiHom-pre-Lie algebras $(A,\ast_A, \alpha_1,\a_2)$ and $(B, \ast_B,\beta_1,\b_2)$. In view of Proposition \ref{propa1}, the linear maps $l_{\ast_{A}}-(r_{\ast_{A}}\circ\a_1\a_2^{-1})\b_1^{-1}\b_2:A\rightarrow gl(B)$ and  $l_{\ast_{B}}-(r_{\ast_{B}}\circ\a_1\a_2^{-1})\b_1^{-1}\b_2:B\rightarrow gl(A)$ are representations of the underlying BiHom-Lie algebras $(A,\{\cdot,\cdot\}_A, \alpha_1,\a_2)$ and $(B,\{\cdot,\cdot\}_B,\beta_1,\b_2)$, respectively. Therefore, (\ref{Lie1}) is equivalent to (\ref{Lie11})-(\ref{Lie12}) and (\ref{Lie2}) is equivalent to (\ref{Lie13})-(\ref{Lie14}).
\end{proof}
Now, we recall the definition of BiHom-dendriform algebra \cite{LiuMakhloufMeninPanaite} and their notions of bimodule and matched pair given in \cite{double}.
\begin{defn}
A BiHom-dendriform algebra is a quintuple $(A, \prec, \succ,
\alpha_1,\a_2)$ consisting of a vector space $A$ on which the
operations $\prec, \succ: A\otimes A\rightarrow
A,$ and $\alpha_1,\a_2: A\rightarrow A$ are
linear maps such that $\a_1\circ\a_2=\a_2\circ\a_1$ and for all $x,y,z\in A$ the following equalities
are satisfied:
\begin{eqnarray*}
\a_1(x \prec y)&=&\a_1(x)\prec\a_1(y)\cr
\a_2(x \prec y)&=&\a_2(x)\prec\a_2(y)\cr
\a_1(x \succ y)&=&\a_1(x)\succ\a_1(y)\cr
\a_2(x \succ y)&=&\a_2(x)\succ\a_2(y)\cr
(x \prec y)\prec \alpha_2(z) &=& \alpha_1(x)\prec (y \cdot z),\cr (x
\succ y) \prec\alpha_2(z) &=& \alpha_1(x) \succ (y \prec z),\cr
\alpha_1(x)\succ (y \succ z ) &=& ( x \cdot y) \succ \alpha_2(z),
\end{eqnarray*}
where
\begin{eqnarray}\label{associative-dendriform}
x \cdot y = x \prec y + x \succ y.
\end{eqnarray}
\end{defn}
\begin{lem}\label{lem2}\cite{LiuMakhloufMeninPanaite}
Let $(A, \prec, \succ, \alpha_1,\a_2)$ be a BiHom-dendriform
algebra. Then, $(A, \cdot:=\prec+\succ, \alpha_1,\a_2)$ is a BiHom-associative
algebra.
\end{lem}
\begin{defn}
Let $(A, \prec, \succ, \alpha_{1}, \alpha_{2})$ be a BiHom-dendriform algebra, and $V$ be a vector space.
Let $l_{\prec}, r_{\prec}, l_{\succ}, r_{\succ} : A \rightarrow gl(V),$ and $\beta_{1}, \beta_{2}: V \rightarrow V$ be six linear maps. Then, ($ l_{\prec}, r_{\prec}, l_{\succ}, r_{\succ}, \beta_{1}, \beta_{2}, V$) is called a bimodule of $ A $
 if the following equations hold for any $ x, y \in A $ and $v\in V$:
\begin{eqnarray}
l_{\prec}(x \prec y)\beta_{2}(v)&=& l_{\prec}(\alpha_{1}(x))l_{\cdot}(y)v,\\ r_{\prec}(\alpha_{2}(x))l_{\prec}(y)v&=&l_{\prec}(\alpha_{1}(y))r_{\cdot}(x)v,\\
r_{\prec}(\alpha_{2}(y))r_{\prec}(y)v &=& r_{\prec}(x\cdot y)\beta_{1}(v),\\ l_{\prec}(x \succ y)\beta_{2}(v) &=& l_{\succ}(\alpha_{1}(x))l_{\prec}(y)v,\\
r_{\prec}(\alpha_{2}(x))l_{\succ}(y)v &=& l_{\succ}(\alpha_{1}(y))r_{\prec}(x)v,\\ r_{\prec}(\alpha_{2}(x))r_{\succ}(y)v &=& r_{\succ}(y\prec x)\beta_{1}(v),\\
l_{\succ}(x\cdot y)\beta_{2}(v) &=& l_{\succ}(\alpha_{1}(x))l_{\succ}(y)v,\\ r_{\succ}(\alpha_{2}(x))l_{\cdot}(y)v&=& l_{\succ}(\alpha_{1}(y))r_{\succ}(x)v,\\
r_{\succ}(\alpha_{2}(x))r_{\cdot}(y)v &=& r_{\succ}(y \succ x)\beta_{1}(v),\\
\beta_{1}(l_\prec(x)v)&=& l_\prec(\alpha_{1}(x))\beta_{1}(v),\\ \beta_{1}(r_\prec(x)v)&=& r_\prec(\alpha_{1}(x))\beta_{1}(v),\\
\beta_{2}(l_\prec(x)v) &=& l_\prec(\alpha_{2}(x))\beta_{2}(v),\\\beta_{2}(r_\prec(x)v)&=& r_\prec(\alpha_{2}(x))\beta_{2}(v)\\
\beta_{1}(l_\succ(x)v)&=& l_\succ(\alpha_{1}(x))\beta_{1}(v),\\ \beta_{1}(r_\succ(x)v)&=& r_\succ(\alpha_{1}(x))\beta_{1}(v),\\
\beta_{2}(l_\succ(x)v) &=& l_\succ(\alpha_{2}(x))\beta_{2}(v),\\\beta_{2}(r_\succ(x)v)&=& r_\succ(\alpha_{2}(x))\beta_{2}(v)
\end{eqnarray}
where $x\cdot y=x\prec y+x\succ y,~l_\cdot=l_\prec+l_\succ$ and $r_\cdot=r_\prec+r_\succ.$
\end{defn}
\begin{prop}
Let $(l_{\prec}, r_{\prec}, l_{\succ}, r_{\succ}, \beta_{1}, \beta_{2}, V)$ be a bimodule of a BiHom-dendriform algebra $(A,\prec, \succ, \alpha_{1}, \alpha_{2}).$ Then, there exists a BiHom-dendriform algebra structure on the direct sum $A\oplus V $ of the underlying vector spaces of $ A $ and $V$ given by
\begin{eqnarray*}
(x + u) \prec'(y + v) &:=& x \prec y + l_{\prec}(x)v + r_{\prec}(y)u\cr
(x + u) \succ' (y + v) &:=& x \succ y + l_{\succ}(x)v + r_{\succ}(y)u, \cr
(\alpha_1\oplus\beta_1)(x+a)&:=&\alpha_1(x)+\beta_1(a)\cr
(\alpha_2\oplus\beta_2)(x+a)&:=&\alpha_2(x)+\beta_2(a)
\end{eqnarray*}
for all $ x, y \in A, u, v \in V $. We denote it by $ A \times_{l_{\prec},r_{\prec}, l_{\succ}, r_{\succ}, \alpha_{1}, \alpha_{2}, \beta_{1}, \beta_{2}} V$.
\end{prop}
\begin{prop}\label{propa2}\cite{double}
Let ($ l_{\prec}, r_{\prec}, l_{\succ}, r_{\succ},
 \beta_1,\b_2, V$) be a bimodule of  BiHom-dendriform algebra
$(A, \prec, \succ, \alpha_1,\a_2)$. Let $(A, \cdot=\prec+\succ, \alpha_1,\a_2)$ be the BiHom
associative algebra. Then ($l_{\prec}+l_{\succ},
r_{\prec}+r_{\succ},\b_1,\b_2,V$) is a
bimodule of $(A, \cdot, \alpha_1,\a_2)$.
\end{prop}
\begin{proof}
We prove only the axiom (\ref{asss1}). The others being proved similarly. For any $x, y \in A$ and
$v \in V$ , we have
\begin{align*}
&(l_\prec+l_\succ)(x\cdot y)\b_2(v)\\=&(l_\prec +l_\succ)(x\prec y+x\succ y)\b_2(v)\\=&l_\prec(x\prec y)+l_\prec(x\succ y)\b_2(v)+l_\succ(x\cdot y)\b_2(v)\\
=&l_\prec(\a_1(x))(l_\prec+l_\succ)(y)v+l_\succ(\a_1(x))l_\prec(y)v+l_\succ(\a_1(x))l_\succ(y)v\\
=&l_\prec(\a_1(x))(l_\prec+l_\succ)(y)v+l_\succ(\a_1(x))(l_\prec+l_\succ)(y)v\\
=&(l_\prec+l_\succ)(\a_1(x))(l_\prec+l_\succ)(y)v.
\end{align*}
This finishes the proof.
\end{proof}
\begin{thm}
Let $(A,\prec_A,\succ_A,\alpha_1,\alpha_2)$ and $(B,\prec_B,\succ_{B},\beta_1,\beta_2)$ be two BiHom-dendriform algebras. Suppose that there are linear maps $l_{\prec_A},r_{\prec_A},l_{\succ_A},r_{\succ_A}:A\rightarrow gl(B)$
and $l_{\prec_B},r_{\prec_B},l_{\succ_B},r_{\succ_B}:B\rightarrow gl(A)$ such that for all $x,y\in A,~a,b\in B$, the following equalities hold:
\begin{eqnarray}
\label{bieq35}
r_{\prec_{A}}(\alpha_{2}(x))(a \prec_{B} b) = \beta_{1}(a)\prec_{B}( r_{A}(x)b) + r_{\prec_{A}}(l_{B}(x)\beta_{1}(a)),
\\
\label{bieq36}
\begin{array}{l}
l_{\prec_{A}}(l_{\prec_{B}}(x))\beta_{2}(b) + (r_{\prec_{A}}(x)a) \prec_{B}\beta_{2}(b)= \cr
\beta_{1}(a) \prec_{B} (l_{\prec_{A}}(x)b) + r_{\prec_{A}}(r_{\prec_{B}}(b)x)\beta_{1}(a),
\end{array} \\
\label{bieq37}
l_{\prec_{A}}(\alpha_{1}(x))(a \ast_{B} b) = (l_{\prec_{A}}(x)a) \ast_{B} \beta_{2}(b) +
 l_{\prec_{A}}(r_{\prec_{A}}(a)x)\beta_{2}(b),
\end{eqnarray}
\begin{eqnarray} \label{bieq38}
r_{\prec_{A}}(\alpha_{2}(x))(a \succ_{B} b) = r_{\succ_{A}}(l_{\prec_{B}}(b)x)\beta_{1}(a) +
\beta_{1}(a)\succ_{B} (r_{\prec_{A}}(x)b), \\
\label{bieq39}
\begin{array}{ll}
l_{\prec_{A}}(l_{\succ_{B}}(a)x)\beta_{2}(b) & + (r_{\succ_{A}}(x)a) \prec_{B}\beta_{2}(b)=\cr
& \beta_{1}(a)\succ_{B} (l_{\prec_{A}}(x)b) + r_{\succ_{A}}(r_{\prec_{B}}(b)x)\beta_{1}(a)
\end{array} \\
\label{bieq40}
l_{\succ_{A}}(\alpha_{1}(x))(a \prec_{B} b) = ( l_{\succ_{A}}(x)a) \prec_{B}\beta_{2}(b) +
 l_{\prec_{A}}(r_{\succ_{B}}(a)x)\beta_{2}(b),
\\
\label{bieq41}
r_{\succ_{A}}(\alpha_{2}(x))(a \ast_{B} b)= \beta_{1}(a)\succ_{B} (r_{\succ_{A}}(x)b) +
 r_{\succ_{A}}(l_{\succ_{B}}(b)x)\beta_{1}(a),
\\
\label{bieq42}
\begin{array}{ll}
\beta_{1}(a)\succ_{B} (l_{\succ_{A}}(x)b) &+ r_{\succ_{A}}(r_{\succ_{B}}(b)x)\beta_{1}(a)=\cr
&l_{\succ_{A}}(l_{B}(a)x)\beta_{2}(b) + (r_{A}(x)a) \succ_{B}\beta_{2}(b),
\end{array} \\
\label{bieq43}
l_{\succ_{A}}(\alpha_{1}(x))(a \succ_{B} b) = (l_{A}(x)a) \succ_{B}\beta_{2}(b) + l_{\succ_{A}}(r_{B}(a)x)\beta_{2}(b),
\\
\label{bieq44}
r_{\prec_{B}}(\beta_{2}(a))(x \prec_{A} y) = \alpha_{1}(x)\prec_{A} (r_{B}(a)y) + r_{\prec_{B}}(l_{A}(y)a)\alpha_{1}(x),
\\
\label{bieq45}
\begin{array}{ll}
l_{\prec_{B}}(l_{\prec_{A}}(x)a)\alpha_{2}(y) &+ (r_{\prec_{B}}(a)x) \prec_{A}\alpha_{2}(y)=\cr
&\alpha_{1}(x)\prec_{A} (l_{B}(a)y) + r_{\prec_{B}}(r_{A}(y)a)\alpha_{1}(x),
\end{array} \\
\label{bieq46}
l_{\prec_{B}}(\beta_{1}(a))(x \ast_{A} y) = (l_{\prec_{B}}(a)x) \prec_{A}\alpha_{2}(y) +
l_{\prec_{B}}(r_{\prec_{A}}(x)a)\alpha_{2}(y),
\\
\label{bieq47}
r_{\prec_{B}}(\beta_{2}(a))(x \succ_{A} y) = r_{\succ_{B}}(l_{\prec_{B}}(y)a)\alpha_{1}(x) +
 \alpha_{1}(x)\succ_{A} (r_{\prec_{B}}(a)y),
\\
\label{bieq48}
\begin{array}{ll}
l_{\prec_{B}}(l_{\succ_{A}}(x)a)\alpha_{2}(y) &+ (r_{\succ_{B}}(a)x) \prec_{A}\alpha_{2}(y)=\cr
&\alpha_{1}(x)\succ_{A} (l_{\prec_{B}}(a)y) + r_{\succ_{B}}(r_{\prec_{A}}(y)a)\alpha_{1}(x),
\end{array} \\
\label{bieq49}
l_{\succ_{B}}(\beta_{1}(a))(x \prec_{A} y) = (l_{\succ_{B}}(a)x) \prec_{A}\alpha_{2}(y) +
l_{\prec_{B}}(r_{\succ_{A}}(x)a)\alpha_{2}(y),
\\
\label{bieq50}
 r_{\succ_{B}}(\beta_{2}(a))(x \ast_{A} y)= \alpha_{1}(x)\succ_{A} (r_{\succ_{B}}(a)y) +
r_{\succ_{B}}(l_{\succ_{A}}(y)a)\alpha_{1}(x),
\\
\label{bieq51}
\begin{array}{ll}
\alpha_{1}(x)\succ_{A} (l_{\succ_{B}}(a)y) & + r_{\succ_{B}}(r_{\succ_{A}}(y)a)\alpha_{1}(x)=\cr
& l_{\succ_{B}}(l_{A}(x)a)\alpha_{2}(y) + (r_{B}(a)x) \succ_{A}\alpha_{2}(y),
\end{array} \\
\label{bieq52}
l_{\succ_{B}}(\beta_{1}(a))(x \succ_{A} y) = (l_{B}(a)x) \succ_{A}\alpha_{2}(y) +
l_{\succ_{B}}(r_{A}(x)a)\alpha_{2}(y),
\end{eqnarray}
where \begin{align*}x\cdot_A y=x\prec_A y + x \succ_A y,~~~l_{\cdot_A} = l_{\prec_A} +
l_{\succ_A},~~~r_{\cdot_A} = r_{\prec_A} + r_{\succ_A},\\
a\cdot_B b= a\prec_B b + a \succ_B b,~~~l_{\cdot_B} = l_{\prec_B} +
l_{\succ_B},~~~r_{\cdot_B} = r_{\prec_B} + r_{\succ_B}.
\end{align*}
Then $(A,B,l_{\prec_A},r_{\prec_A},l_{\succ_A},r_{\succ_A},\beta_1,\beta_2,l_{\prec_B},r_{\prec_B},l_{\succ_B},r_{\succ_B},\alpha_1,\alpha_2)$ is called a matched pair of BiHom-dendriform algebras. In this case, there exists a BiHom-dendriform algebra structure on the direct sum
$A\oplus B$ of the underlying vector spaces of $A$ and $B$ given by
$$\begin{array}{llllll}
(x + a) \prec(y + b)&:=&x \prec_A y + (l_{\prec_A}(x)b + r_{\prec_A}(y)a)+a \prec_B b + (l_{\prec_B}(a)y + r_{\prec_B}(b)x),\cr
(x + a) \succ (y + b)&:=&x \succ_A y + (l_{\succ_A}(x)b + r_{\succ_A}(y)a)+a \succ_B b + (l_{\succ_B}(a)y + r_{\succ_B}(b)x),\cr
(\alpha_{1}\oplus\beta_{1})(x + a)&:=&\alpha_{1}(x) + \beta_{1}(a),\cr (\alpha_{2}\oplus\beta_{2})(x + a)&:=&\alpha_{2}(x) + \beta_{2}(a).
\end{array}$$
\end{thm}

We denote this BiHom-dendriform algebra by $A\bowtie^{l_A,r_A,\beta_1,\beta_2}_{l_B,r_B,\alpha_1,\alpha_2}B$.\\
\begin{prop}\cite{double}\label{Matchedd2}
Let $(A, B, l_{\prec_{A}}, r_{\prec_{A}}, l_{\succ_{A}}, r_{\succ_{A}}, \beta_1,\b_2,
 l_{\prec_{B}}, r_{\prec_{B}}, l_{\succ_{B}}, r_{\succ_{B}}, \alpha_1,\a_2) $ be a matched pair of a BiHom-dendriform algebras $(A, \prec_{A}, \succ_{A}, \alpha_1,\a_2)$ and $(B, \prec_{B}, \succ_{B},\beta_1,\b_2)$.
Then, $(A, B, l_{\prec_{A}} + l_{\succ_{A}}, r_{\prec_{A}} + r_{\succ_{A}},\b_1,\b_2,
l_{\prec_{B}} + l_{\succ_{B}},  r_{\prec_{B}} + r_{\succ_{B}} ,\a_1,\a_2)$ is a matched pair of the associated
BiHom-associative algebras $(A,\cdot_{A}=\prec_A+\succ_A, \alpha_1,\a_2)$ and $(B, \cdot_{B}=\prec_B+\succ_B,\beta_1,\b_2)$.
\end{prop}
\begin{proof}
Let $(A, B, l_{\prec_{A}}, r_{\prec_{A}}, l_{\succ_{A}}, r_{\succ_{A}}, \beta_1,\b_2,
 l_{\prec_{B}}, r_{\prec_{B}}, l_{\succ_{B}}, r_{\succ_{B}}, \alpha_1,\a_2)$ be a matched pair of a BiHom-dendriform algebras $(A, \prec_{A}, \succ_{A}, \alpha_1,\a_2)$ and $(B, \prec_{B}, \succ_{B},\beta_1,\b_2)$. In view of Proposition \ref{propa2}, the linear maps $l_{\prec_{A}} + l_{\succ_{A}}, r_{\prec_{A}} + r_{\succ_{A}}:A\rightarrow gl(B)$ and  $l_{\prec_{B}} + l_{\succ_{B}},  r_{\prec_{B}} + r_{\succ_{B}}:B\rightarrow gl(A)$ are bimodules of the underlying BiHom-Lie algebras $(A,\cdot_A, \alpha_1,\a_2)$ and $(B,\cdot_B,\beta_1,\b_2)$, respectively. Therefore, (\ref{3})-(\ref{5}) are equivalents to (\ref{bieq35})-(\ref{bieq43}) and (\ref{6})-(\ref{8}) are equivalents to (\ref{bieq44})-(\ref{bieq52}).
\end{proof}
Now, we introduce the definition of noncommutative BiHom-pre-Poisson algebra and we give some results.
\begin{defn}\label{def pre-poiss}
A noncommutative BiHom-pre-Poisson algebra is a $6$-tuple
$(A,\prec,\succ,\ast,\alpha_1,\a_2)$ such that $(A,\prec,\succ,\alpha_1,\a_2)$ is a
BiHom-dendriform algebra and $(A,\ast,\alpha_1,\a_2)$ is a BiHom-pre-Lie algebra
satisfying the following compatibility conditions:
\begin{eqnarray}
    \label{eq:pre-Poisson 1}(\a_2(x)\ast \a_1(y)-\a_2(y)\ast \a_1(x))\prec \alpha_2(z)&=&\a_1\a_2(x)\ast(\a_1(y)\prec z)-\a_1\a_2(y)\prec(\a_1(x)\ast z),\nonumber\\
    &&\\
    \label{eq:pre-Poisson 2} \a_2(x)\succ(\a_1\a_2(y)\ast \a_1(z)-\a_2(z)\ast \a_1^{2}(y))&=&\a_1\a_2^{2}(y)\ast(x\succ \a_1(z))-(\a_2^{2}(y)\ast x)\succ \a_1\a_2(z),\nonumber\\
    &&\\
    \label{eq:pre-Poisson 3} (\a_2(x)\prec \a_1(y)+\a_2(x)\succ \a_1(y))\ast \alpha_2(z)&=&(\a_2(x)\ast \a_1(z))\succ \alpha_2(y)+\alpha_1\a_2(x)\prec(\a_1(y)\ast z).\nonumber\\
    &&
  \end{eqnarray}
\end{defn}
\begin{thm} \label{Yaupre-poiss}
Let $(A, \prec , \succ ,\ast)$ be a noncommutative pre-Poisson
algebra \cite{Baii} and $\alpha_1,\a_2 :A\rightarrow A$ be two morphisms of $A$ such that $\a_1\a_2=\a_2\a_1$. Then $A_{\alpha_1,\a_2}:=(A, \prec _{\alpha_1,\a_2}=\prec\circ(\a_1\otimes\a_2), \succ_{\alpha_1,\a_2}=\succ\circ(\a_1\otimes\a_2),\ast_{\alpha_1,\a_2}=\ast\circ(\a_1\otimes\a_2), \alpha_1,\a_2)$ is a noncommutative BiHom-pre-Poisson algebra, called
the Yau twist of $A$. Moreover, assume that $(A', \prec ', \succ
',\ast')$ is another noncommutative pre-Poisson algebra and $\alpha'_1,\a'_2:A'\rightarrow
A'$
be a two commuting noncommutative pre-Poisson algebra morphisms. Let $f:A\rightarrow A'$ be a
pre-Poisson algebra morphism satisfying $f\circ \alpha_1 =\alpha'_1\circ f$ and $f\circ \alpha_2 =\alpha'_2\circ f$. Then $f:A_{\alpha_1,\a_2 }\rightarrow A'_{\alpha'_1,\a'_2}$ is a
noncommutative BiHom-pre-Poisson algebra morphism.
\end{thm}
\begin{proof}
We shall only prove relation (\ref{eq:pre-Poisson 1}) the others being proved analogously. Then, for
any $x, y, z \in A$,
\begin{align*}
&(\a_2(x)\ast_{\a_1,\a_2} \a_1(y)-\a_2(y)\ast_{\a_1,\a_2} \a_1(x))\prec_{\a_1,\a_2} \alpha_2(z)\\
=&(\a_1\a_2(x)\ast \a_1\a_2(y)-\a_1\a_2(y)\ast \a_1\a_2(x))\prec_{\a_1,\a_2} \alpha_2(z)\\
=&(\a_1^{2}\a_2(x)\ast \a_1^{2}\a_2(y)-\a_1^{2}\a_2(y)\ast \a_1^{2}\a_2(x))\prec \alpha_2^{2}(z)\\
=&\a_1^{2}\a_2(x)\ast(\a_1^{2}\a_2(y)\prec \a_2^{2}(z))-\a_1^{2}\a_2(y)\prec(\a_1^{2}\a_2(x)\ast \a_2^{2}(z))\\
=&\a_1\a_2(x)\ast_{\a_1,\a_2}(\a_1(y)\prec_{\a_1,\a_2} z)-\a_1\a_2(y)\prec_{\a_1,\a_2}(\a_1(x)\ast_{\a_1,\a_2} z).
\end{align*}
For the second assertion, we have
\begin{align*}
    f(x\prec_{\a_1,\a_2} y)=&f(\a_1(x)\prec\a_2(y))\\
    =&f(\a_1(x))\prec' f(\a_2(y))\\
    =&\a'_1 f(x)\prec'\a'_2 f(y)\\
    =&f(x)\prec'_{\a_1,\a_2} f(y).
\end{align*}
Similarly, we have $ f(x\succ_{\a_1,\a_2}y)=f(x)\succ'_{\a'_1,\a'_2} f(y)$ and $ f(x\ast_{\a_1,\a_2}y)=f(x)\ast'_{\a'_1,\a'_2} f(y)$.
This completes the proof.
\end{proof}
\begin{prop}\label{isma}
More generally, let $(A, \prec , \succ ,\ast, \alpha_1,\a_2)$ be a
noncommutative BiHom-pre-Poisson algebra and $\alpha'_1,\a'_2: A\rightarrow A$ be a two noncommutative BiHom-pre-Poisson algebra morphisms such that any two of the maps $\a_1,\a_2,\a'_1,\a'_2$
commute.
Then  $(A, \prec _{\alpha'_1,\a'_2}, \succ_{\alpha'_1,\a'_2},\ast_{\alpha'_1,\a'_2},\alpha_1 \circ \alpha'_1,\alpha_2 \circ \alpha'_2)$ is a noncommutative BiHom-pre-Poisson
algebra.
\end{prop}
\begin{cor}
Let $(A, \prec , \succ,\ast,\alpha_1,\a_2)$ be a noncommutative BiHom-pre-Poisson
algebra and $n\in\mathbb{N}^{\ast}$. Then
\begin{enumerate}
\item
The $nth$ derived noncommutative BiHom-pre-Poisson algebra of type $1$ of $A$ is
defined by
$$A_{1}^{n}=(A,\prec^{(n)}=\prec\circ(\a_1^{n}\otimes\a_2^{n}),\succ^{(n)}=\succ\circ(\a_1^{n}\otimes\a_2^{n})
,\ast^{(n)}=\ast\circ(\a_1^{n}\otimes\a_2^{n}),\alpha_1^{n+1},\alpha_2^{n+1}).$$
\item
The $nth$ derived noncommutative BiHom-pre-Poisson algebra of type $2$ of $A$ is
defined by
\begin{align*}&A_{2}^{n}=(A,\prec^{(2^n-1)}=\prec\circ(\a_1^{2^{n}-1}\otimes\a_2^{2^{n}-1}),\succ^{(2^n-1)}=\succ\circ(\a_1^{2^{n}-1}\otimes\a_2^{2^{n}-1}),\\&\ast^{(2^n-1)}=\ast\circ(\a_1^{2^{n}-1}\otimes\a_2^{2^{n}-1}),\alpha_1^{2^n},\alpha_2^{2^n}).\end{align*}
\end{enumerate}
\end{cor}
\begin{proof}
Apply Theorem \ref{isma} with $\alpha'_1=\alpha_1^{n},~\a'_2=\alpha_2^{n}$ and
$\alpha'_1=\alpha_1^{2^{n}-1},~\a'_2=\alpha_2^{2^{n}-1}$ respectively.
\end{proof}
\begin{thm}
Let $(A,\prec,\succ,\ast,\a_1,\a_2)$ be a regular noncommutative BiHom-pre-Poisson
algebra. Then $(A,\cdot,\{\cdot,\cdot\},\a_1,\a_2)$ is a noncommutative
BiHom-Poisson algebra with
$$x\cdot y=x\prec y+x\succ y,~~and ~~\{x, y\}=x\ast y-\a_1^{-1}\a_2(y)\ast \a_1\a_2^{-1}(x),$$
for any $x,y \in A$. We say that $(A,\cdot,\{\cdot,\cdot\},\a_1,\a_2)$ is
the sub-adjacent noncommutative BiHom-Poisson algebra of $(A,\prec,\succ,\ast,\a_1,\a_2)$ and
denoted by $A^{c}$.
\end{thm}
\begin{proof}
By Lemma \ref{lem1} and Lemma \ref{lem2}, we deduce that
$(A,\cdot,\a_1,\a_2)$ is a BiHom-associative algebra and
$(A,\{\cdot,\cdot\},\a_1,\a_2)$ is a BiHom-Lie algebra. Now, we show the
BiHom-Leibniz identity
\begin{align*}
   & \{\a_1\a_2(x),y\cdot z\}-\{\a_2(x),y\}\cdot\a_2(z)-\a_2(y)\cdot\{\a_1(x),z\}\\
    =&\{\a_1\a_2(x),y\prec z+ y\succ z\}-\{\a_2(x),y\}\prec\a_2(z)-\{\a_2(x),y\}\succ\a_2(z)\\&-\a_2(y)\prec\{\a_1(x),z\}-\a_2(y)\succ\{\a_1(x),z\}\\
    =&\a_1\a_2(x)\ast(y\prec z)+\a_1\a_2(x)\ast(y\succ z)-\a_1^{-1}\a_2(y\prec z)\ast\a_1^{2}(x)\\&-\a_1^{-1}\a_2(y\succ z)\ast\a_1^{2}(x)-(\a_2(x)\ast y)\prec\a_2(z)+(\a_1^{-1}\a_2(y)\ast\a_1(x)))\prec\a_2(z)\\&-(\a_2(x)\ast y)\succ\a_2(z)-(\a_1^{-1}\a_2(y)\ast\a_1(x))\succ\a_2(z)-\a_2(y)\prec(\a_1(x)\ast z)\\&+\a_2(y)\prec(\a_1^{-1}\a_2(z)\ast\a_1^{2}\a_2^{-1}(x))-\a_2(y)\succ(\a_1(x)\ast z)+\a_2(y)\succ(\a_1^{-1}\a_2(z)\ast\a_1^{2}\a_2^{-1}(x))\\
    =&\Big(\a_1\a_2(x)\ast(y\prec z)-\a_2(y)\prec(\a_1(x)\ast z)-(\a_2(x)\ast y-\a_1^{-1}\a_2(y)\ast\a_1(x))\prec(z)\Big)\\&+\Big(\a_1\a_2(x)\ast(y\succ z)-(\a_2(x)\ast y)\succ\a_2(z)-\a_2(y)\succ(\a_1(x)\ast z-\a_1^{-1}\a_2(z)\ast\a_1^{2}\a_2^{-1}(x)\Big)\\
    &+\Big((\a_1^{-1}\a_2(y)\ast \a_1(x))\succ\a_2(z)+\a_2(y)\prec(\a_1^{-1}\a_2(z)\ast\a_1^{2}\a_2^{-1}(x))\\&-(\a_1^{-1}\a_2(y)\succ\a_1^{-1}\a_2(z)\a_1^{-1}\a_2(y)\succ\a_1^{-1}\a_2(z))\ast\a_1^{2}(x)\Big)\\
    =&0+0+0=0~~(by~(\ref{eq:pre-Poisson 1})-(\ref{eq:pre-Poisson 3}))
\end{align*}
which implies that $(A,\cdot,\{\cdot,\cdot\},\a_1,\a_2)$ is a
noncommutative BiHom-Poisson algebra.
\end{proof}
The relation existing between a noncommutative BiHom-Poisson
algebra and noncommutative BiHom-pre-Poisson
algebra, as illustrated by the following diagram:
$$
\xymatrix{ \mbox{BiHom-dendriform alg+ BiHom-pre-Lie alg} \ar[rr] \ar[dd]^{x\ast y-\a_1^{-1}\a_2(y)\a_1\a_2^{-1}(x)}_{x\prec y + x\succ y} && \mbox{BiHom-pre-Poisson alg} \ar[dd]^{x\ast y-\a_1^{-1}\a_2(y)\a_1\a_2^{-1}(x)}_{x\prec y + x\succ y}\\
&& \\
\mbox{BiHom-associative alg+BiHom-Lie alg} \ar[rr] && \mbox{BiHom-Poisson alg.}
}
$$

In the following we introduce the notions of bimodule and matched pair of noncommutative BiHom-pre-Poisson algebras and related relevant properties are also given
\begin{defn}\label{def bim}
Let $(A, \prec, \succ,\ast, \alpha_1,\a_2)$ be a noncommutative BiHom-pre-Poisson algebra. A bimodule of $A$ is a $9$-tuple
($ l_{\prec}, r_{\prec}, l_{\succ}, r_{\succ},l_{\ast}, r_{\ast},
\beta_1,\b_2,V$) such that ($ l_{\ast}, r_{\ast},
\beta_1,\b_2,V$) is a bimodule of the BiHom-pre-Lie algebra $(A, \ast, \alpha_1,\a_2)$ and ($ l_{\prec}, r_{\prec}, l_{\succ}, r_{\succ},
\beta_1,\b_2,V$) is a bimodule of the
BiHom-dendriform algebra $(A, \prec, \succ, \alpha_1,\a_2)$
satisfying for all $ x, y \in A$ and $v\in V$:
\begin{eqnarray}
&&l_\prec(\{\a_2(x),\a_1(y)\})\b_2(v)=l_\ast(\a_1\a_2(x))l_\prec(\a_1(y))v-l_\prec(\a_1\a_2(y))l_\ast(\a_1(x))v,\label{1001}\\
&&r_\prec(\a_2(x))\rho(\a_2(y))\b_1(v)=l_\ast(\a_1\a_2(y))r_\prec(x)\b_1(v)-r_\prec(\a_1(y)\ast x)\b_1\b_2(v),\label{1002}\\
&&-r_\prec(\a_2(x))\rho(\a_2(y))\b_1(v)=r_\ast(\a_1(y)\prec x)\b_1\b_2(v)-l_{\prec}(\a_1\a_2(y))r_\ast(x)\b_1(v),\label{1003}\\
&&l_\succ(\a_2(x))\rho(\a_1\a_2(y))\b_1(v)=l_\ast(\a_1\a_2^{2}(y))l_\succ(x)\b_1(v)-l_\succ(\a_2^{2}(y)\ast z)\b_1\b_2(v),\label{1004}\\
&&r_\succ(\{\a_1\a_2(x),\a_1(y)\})\b_2(v)=l_\ast(\a_1\a_2^{2}(x))r_\succ(\a_1(y))v-r_\succ(\a_1\a_2(y))l_\ast(\a_2^{2}(y))v,\label{1005}\\
&&-l_\succ(\a_2(x))\rho(\a_2(y))\b_1^{2}(v)=r_{\ast}(x\succ\a_1(y))\b_1\b_2^{2}(v)-r_\succ(\a_1\a_2(y))r_\ast(x)\b_2^{2}(v),\label{1006}\\
&&l_\ast(\a_2(x)\cdot\a_1(y))\b_2(v)=r_\succ(a_2(y))l_\ast(\a_2(x))\b_1(v)+l_\prec(\a_1\a_2(x))l_\ast(\a_1(y))v\label{1007},\\
&&r_\ast(\a_2(x))l_\cdot(\a_2(y))\b_1(v)=l_\succ(\a_2(y)\ast\a_1(x))\b_2(v)+l_\prec(\a_1\a_2(y))r_\ast(x)\b_1(v),\label{1008}\\
&&r_\ast(\a_2(x))r_\cdot(\a_1(y))\b_2(v)=r_\succ(\a_2(y))r_\ast(\a_1(x))\b_2(v)+r_\prec(\a_1(y)\ast x)\b_1\b_2(v),\label{1009}
\end{eqnarray}where
$$\begin{array}{lllllll} x \cdot y = x \prec y + x \succ y,~l_{\cdot} = l_{\prec} +
l_{\succ},~r_{\cdot} = r_{\prec} + r_{\succ},
\\
 \{\a_2(x),\a_1(y)\} = \a_2(x)
\ast \a_1(y) - \a_2(y) \ast \a_1(x),\\ (\rho\circ\a_2)\b_1 = (l_{\ast}\circ\a_2)\b_1 - (r_{\ast}\circ \a_1)\b_2.
\end{array}$$

\end{defn}
\begin{prop}
Let $(l_{\prec}, r_{\prec}, l_{\succ}, r_{\succ},l_{\ast}, r_{\ast},
\b_1,\b_2, V)$ be a bimodule of a noncommutative BiHom-pre-Poisson algebra $(A,\prec,
\succ,\ast, \alpha_1,\a_2).$ Then, there exists a noncommutative BiHom-pre-Poisson algebra structure on the direct sum $A\oplus V $ of the underlying
vector spaces of $A$ and $V$ given by
\begin{eqnarray*}
(x + u) \prec' (y + v) &:=& x \prec y + l_{\prec}(x)v +
r_{\prec}(y)u, \cr (x + u) \succ' (y + v) &:=& x \succ y +
l_{\succ}(x)v + r_{\succ}(y)u,\cr (x + u) \ast' (y + v) &:=& x \ast y
+ l_{\ast}(x)v + r_{\ast}(y)u,\\
(\a_1\oplus\b_1)(x+u)&:=&\a_1(x)+\b_1(u),\\
(\a_2\oplus\b_2)(x+u)&:=&\a_2(x)+\b_2(u),
\end{eqnarray*}
for all $ x, y \in A, u, v \in V $. We denote it by $
A\times_{l_{\prec},r_{\prec}, l_{\succ}, r_{\succ}, l_{\ast},
r_{\ast},\alpha_1,\a_2,\beta_1,\b_2} V$.
\end{prop}

\begin{proof}
We prove only the axiom (\ref{eq:pre-Poisson 1}). The axioms (\ref{eq:pre-Poisson 2}), (\ref{eq:pre-Poisson 3}) being proved similarly. For any
$x_{1},x_{2},x_{3}\in A$ and $v_1, v_2, v_3\in V$, we have
\begin{align*}
    &((\a_2+\b_2)(x_1+v_1)\ast'(\a_1+\b_1)(x_2+v_2))\prec'(\a_2+\b_2)(x_3+v_3)\\&-((\a_2+\b_2)(x_2+v_2)\ast'(\a_1+\b_1)(x_1+v_1))\prec'(\a_2+\b_2)(x_3+v_3)\\
    =&((\a_2(x_1)\ast\a_1(x_2))+l_\ast(\a_2(x_1))\b_1(v_2)+r_\ast(\a_1(x_2))\b_2(v_1))\prec'(\a_2+\b_2)(x_3+v_3)\\
    &-(\a_2(x_2)\ast\a_1(x_1)+l_\ast(\a_2(x_2)\b_1(v_1)+r_\ast(\a_1(x_1))\b_2(v_2))\prec'(\a_2+\b_2)(x_3+v_3)\\
    =&(\a_2(x_1)\ast\a_1(x_2))\prec\a_2(x_3)+l_\prec(\a_2(x_1)\ast\a_1(x_2))\b_2(v_3)+r_\prec(\a_2(x_3))l_\ast(\a_2(x_2))\b_1(v_1)\\
    &-r_\prec(\a_2(x_3))r_\ast(\a_1(x_1))\b_2(v_2)-(\a_2(x_2)\ast\a_1(x_1))\prec\a_2(x_3)-l_\prec(\a_2(x_2)\ast\a_1(x_1))\b_2(v_3)\\&+r_\prec(\a_2(x_3))l_\ast(\a_2(x_2))\b_1(v_1)-r_\prec(\a_2(x_3))r_\ast(\a_1(x_1))\b_2(v_2)\\
    =&\{\a_2(x_1),\a_1(x_2)\}+l_\prec(\{\a_2(x_1),\a_1(x_2)\})\b_2(v_3)\\
    &+r_\prec(\a_2(x_3))\rho(\a_2(x_2))\b_1(v_1)+r_\prec(\a_2(x_3))\rho(\a_2(x_2))\b_1(v_1).
\end{align*}
On the other hand ,
\begin{align*}
    &(\a_1+\b_1)(\a_2+\b_2)(x_1+v_1)\ast'(\a_1(x_2+v_2)\prec(x_3+v_3))\\
    &-(\a_1+\b_1)(\a_2+\b_2)(x_2+v_2)\prec'((\a_1+\b_1)(x_1+v_1)\ast'(x_3+v_3))\\
    =&(\a_1\a_2(x_1)+\b_1\b_2(v_1))\ast'(\a_1(x_2)\prec x_3+l_\prec(\a_1(x_2))v_3+r_\prec(x_3)\b_1(v_1)\\
    &-(\a_1\a_2(x_2)+\b_1\b_2(v_2))\prec'(\a_1(x_1)\ast x_3+l_\ast(\a_1(x_1))v_3+r_\ast(x_3)\a_1(v_1))\\
    =&\a_1\a_2(x_1)\ast\a_1(x_2)\prec x_3)+l_\ast(\a_1\a_2(x_1))l_\prec(\a_1(x_2))v_3+l_\ast(\a_1\a_2(x_1))r_\prec(x_3)\b_1(v_1)\\&+r_\ast(\a_1(x_2)\prec x_3)\b_1\b_2(v_1)-\a_1\a_2(x_2)\prec(\a_1(x_1)\ast x_3)-l_\prec(\a_1\a_2(x_2))l_\ast(\a_1(x_1))v_3\\&-l_\prec(\a_1\a_2(x_2))r_\ast(x_3)\a_1(v_1)-r_\prec(\a_1(x_1)\ast x_3)\b_1\b_2(v_2).
\end{align*}
By equations (\ref{eq:pre-Poisson 1}), (\ref{1001})-(\ref{1003}) we deduce that
\begin{align*}
&((\a_2+\b_2)(x_1+v_1)\ast'(\a_1+\b_1)(x_2+v_2))\prec'(\a_2+\b_2)(x_3+v_3)\\&-((\a_2+\b_2)(x_2+v_2)\ast'(\a_1+\b_1)(x_1+v_1))\prec'(\a_2+\b_2)(x_3+v_3)\\
=&(\a_1+\b_1)(\a_2+\b_2)(x_1+v_1)\ast'(\a_1(x_2+v_2)\prec(x_3+v_3))
   \\& -(\a_1+\b_1)(\a_2+\b_2)(x_2+v_2)\prec'((\a_1+\b_1)(x_1+v_1)\ast(x_3+v_3))
    \end{align*}
\end{proof}
There is an example of bimodule of noncommutative BiHom-pre-Poisson algebra
\begin{ex}
Let $(A,\prec, \succ,\ast,\alpha_1,\a_2)$ be a noncommutative BiHom-pre-Poisson algebra.
Then $(L_{\prec},R_{\prec},L_{\succ},R_{\succ},L_{\ast},R_{\ast},\alpha_1,\a_2,A)$ is called a regular
bimodule of $A$, where $L_{\prec}(x)y=x\prec y,~R_{\prec}(x)y=y\prec
x,~L_{\succ}(x)y=x\succ y,~R_{\succ}(x)y=y\succ x$ and
$L_{\ast}(x)y=x\ast y,~R_{\ast}(x)y=y\ast x$, for all $x,y\in A$.
\end{ex}
\begin{prop}
If $f:(A,\prec_1, \succ_1,\ast_1,\alpha_1,\a_2)\longrightarrow(A',\prec_2,
\succ_2,\ast_2,\beta_1,\b_2)$ is a morphism of noncommutative BiHom-pre-Poisson algebra,
then
$(l_{\prec_1},r_{\prec_1},l_{\succ_1},r_{\succ_1},l_{\ast_1},r_{\ast_1},\beta_1,\b_2,A')$
becomes a bimodule of $A$ via $f$, i.e, $l_{\prec_1}(x)y=f(x)\prec_2
y,~r_{\prec_1}(x)y=y \prec_2 f(x),~l_{\succ_1}(x)y=f(x)\succ_2
y,~r_{\succ_1}(x)y=y \succ_2 f(x)$ and $l_{\ast_1}(x)y=f(x)\ast_2
y,~r_{\ast_1}(x)y=y \ast_2 f(x)$ for all $(x,y)\in A\times A'$.
\end{prop}
\begin{proof}
We prove only the axiom (\ref{1007}). The others being proved similarly. For any
$x,y\in A$ and $z\in A'$, we have
\begin{align*}
    &l_{\ast_1}(\a_2(x)\cdot_1\a_1(y))\b_2(z)\\
    =&f(\a_2(x)\cdot_1\a_1(y))\ast_2\b_2(z))\\=&(\b_2 f(x)\cdot_2\b_1 f(y))\ast_2\b_2(z)\\
   =&(\b_2 f(x)\ast_2\b_1(z))\succ_2\b_2 f(y)+\b_1\b_2 f(x)\prec_2(\b_1 f(y)\ast_2 z) ~~(by~(\ref{1001}))\\
   =&(f(\a_2(x))\ast_2\b_1(z))\succ_2 f(\a_2(y))+f(\a_1\a_2(x))\prec_2(f(\a_1(y))\ast_2 z)\\
   =&r_{\succ_1}(\a_2(y))(f(\a_2(x))\ast_2\b_1(z))+l_{\prec_1}(\a_1\a_2(x))(f(\a_1(y))\ast_2 z)\\
    =&r_{\succ_1}(\a_2(y))l_{\ast_1}(\a_2(x))\b_1(z)+l_{\prec_1}(\a_1\a_2(x))l_{\ast_1}(\a_1(y))z.
\end{align*}
This finishes the proof.
\end{proof}
\begin{cor}\label{coroo}
Let ($ l_{\prec}, r_{\prec}, l_{\succ}, r_{\succ},l_{\ast},
r_{\ast}, \beta_1,\b_2, V$) be a bimodule of a regular noncommutative BiHom-pre-Poisson algebra
$(A, \prec, \succ,\ast, \alpha_1,\a_2)$ such that $\b_1$ is bijective. Let $(A, \cdot,\{\cdot,\cdot\}, \alpha_1,\a_2)$ be the
subadjacent of $(A, \prec, \succ,\ast, \alpha_1,\a_2)$. Then ($l_{\prec}+l_{\succ},
r_{\prec}+r_{\succ},l_\ast-(r_\ast\circ\a_1\a_2^{-1})\b_1^{-1}\b_2,\b_1,\b_2,V$) is a
representation of $(A, \cdot,\{\cdot,\cdot\}, \alpha_1,\a_2)$.
\end{cor}
\begin{proof}
It follows from the relation between the noncommutative BiHom-pre-Poisson algebra
and the associated noncommutative BiHom-Poisson algebra. More precisely
by Poposition \ref{propa1} and Poposition \ref{propa2}, we deduce that ($l_\ast-(r_\ast\circ\a_1\a_2^{-1})\b_1^{-1}\b_2,\b_1,\b_2,V$) is a
representation of $(A, \{\cdot,\cdot\}, \alpha_1,\a_2)$ and ($l_{\prec}+l_{\succ},
r_{\prec}+r_{\succ},\b_1,\b_2,V$) is a
bimodule of $(A, \cdot, \alpha_1,\a_2)$. Now, the rest, it is easy ( in a similar way as for Poposition \ref{propa1} and Poposition \ref{propa2}) to verify the axioms (\ref{isma1.1})-(\ref{isma1.3})
\end{proof}
\begin{cor}
Let ($ l_{\prec}, r_{\prec}, l_{\succ}, r_{\succ},l_{\ast},
r_{\ast}, \beta_1,\b_2, V$) be a bimodule of a regular noncommutative BiHom-pre-Poisson algebra
$(A, \prec, \succ,\ast, \alpha_1,\a_2)$ such that $\b_1$ is bijective. Let $(A, \cdot,\{\cdot,\cdot\}, \alpha_1,\a_2)$ be the
subadjacent of $(A, \prec, \succ,\ast, \alpha_1,\a_2)$.
Then
\begin{enumerate}[label=\upshape{\arabic*)}]
\item $(l_{\prec}, r_{\succ},l_\ast-(r_\ast\circ\a_1\a_2^{-1})\b_1^{-1}\b_2, \beta_{1}, \beta_{2}, V) $ is bimodule of $ (A, \cdot,\{\cdot,\cdot\}, \alpha_{1}, \alpha_{2}); $
\item $(l_{\prec} + l_{\succ}, 0, 0,  r_{\prec} + r_{\succ},l_\ast,r_\ast, \beta_{1}, \beta_{2}, V)$ and $(l_{\prec}, 0, 0, r_{\succ},l_\ast,r_\ast, \beta_{1}, \beta_{2}, V)$ are bimodules
 of $(A, \prec, \succ,\ast, \alpha_{1}, \alpha_{2});$
\item  the noncommutative BiHom-pre-Poisson algebras
\begin{eqnarray*}
 A \times_{l_{\prec}, r_{\prec}, l_{\succ}, r_{\succ},l_\ast,r_\ast, \alpha_{1}, \alpha_{2}, \beta_{1}, \beta_{2}}V \mbox{ and }  A \times_{l_{\prec} +  l_{\succ} , 0, 0, r_{\prec} + r_{\succ}, l_\ast,r_\ast,\alpha_{1}, \alpha_{2}, \beta_{1}, \beta_{2}} V
 \end{eqnarray*} have the same associated
noncommutative BiHom-Poisson algebra $$A \times_{l_{\prec} +  l_{\succ}, r_{\prec} + r_{\succ}, l_\ast-(r_\ast\circ\a_1\a_2^{-1})\b_1^{-1}\b_2,\alpha_{1}, \alpha_{2}, \beta_{1}, \beta_{2}} V.$$
 \end{enumerate}
\end{cor}
\begin{proof}
It results from a direct computation.
\end{proof}
The following result gives a construction of a bimodule of a BiHom-pre-Poisson algebra by means of the Yau twist procedure
\begin{thm}\label{mamm1}
Let $(A,\prec, \succ,\ast,\alpha_1,\alpha_2)$ be a noncommutative BiHom-pre-Poisson algebra,\\ $(l_{\prec},r_{\prec},l_{\succ},r_{\succ},l_{\ast},r_{\ast},\beta_1,\beta_2,V)$ be a bimodule of $A$. Let $\alpha'_1,\alpha'_2$ be two endomorphisms of $A$ such that any two of the maps $\alpha_1,\alpha'_1,\alpha_2,\alpha'_2$ commute
and $\beta'_1,~\beta'_2$ be linear maps of $V$ such that any two of the maps $\beta_1,\beta'_1,\beta_2,\beta'_2$ commute. Suppose furthermore that
$$\left\{
   \begin{array}{lllllll}
    \beta'_1\circ l_\prec=(l_\prec\circ\alpha'_1)\beta'_1,~~
     \beta'_2\circ l_\prec=(l_\prec\circ\alpha'_2)\beta'_2,& \\
         \beta'_1\circ l_\succ=(l_\succ\circ\alpha'_1)\beta'_1,~~
     \beta'_2\circ l_\succ=(l_\succ\circ\alpha'_2)\beta'_2,&\\
      \beta'_1\circ l_\ast=(l_\ast\circ\alpha'_1)\beta'_1,~~
     \beta'_2\circ l_\ast=(l_\ast\circ\alpha'_2)\beta'_2,& \\
   \end{array}
 \right.$$
~~and~~
$$\left\{
   \begin{array}{lllllll}
    \beta'_1\circ r_\prec=(r_\prec\circ\alpha'_1)\beta'_1,~~
     \beta'_2\circ r_\prec=(r_\prec\circ\alpha'_2)\beta'_2,& \\
         \beta'_1\circ r_\succ=(r_\succ\circ\alpha'_1)\beta'_1,~~
     \beta'_2\circ r_\succ=(r_\succ\circ\alpha'_2)\beta'_2,&\\
      \beta'_1\circ r_\ast=(r_\ast\circ\alpha'_1)\beta'_1,~~
     \beta'_2\circ r_\ast=(r_\ast\circ\alpha'_2)\beta'_2,&
   \end{array}
 \right.$$

and write $A_{\alpha'_1,\a'_2}$ for the noncommutative BiHom-pre-Poisson algebra $(A,\prec_{\alpha'_1,\alpha'_2}, \succ_{\alpha'_1,\alpha'_2},\ast_{\alpha'_1,\alpha'_2},\alpha_1\alpha'_1,\alpha_2\alpha'_2)$ and
$V_{\beta'_1,\beta'_2}=(\widetilde{l}_{\prec},\widetilde{r}_{\prec},\widetilde{l}_{\succ},
\widetilde{r}_{\succ},\widetilde{l}_{\ast},\widetilde{r}_{\ast},\beta_1\beta'_1,\beta_2\beta'_2,V)$, where
\begin{align*}
&\widetilde{l}_{\prec}=(l_{\prec}\circ\alpha'_1)\beta'_2,~\widetilde{r}_{\prec}=(r_{\prec}\circ\alpha'_2)\beta'_1,~\widetilde{l}_{\succ}=(l_{\succ}\circ\alpha'_1)\beta'_2,\\&
\widetilde{r}_{\succ}=(r_{\succ}\circ\alpha'_2)\beta'_1,~\widetilde{l}_{\ast}=(l_{\ast}\circ\alpha'_1)\beta'_2,~\widetilde{r}_{\ast}=(r_{\ast}\circ\alpha'_2)\beta'_1.
\end{align*}
Then $V_{\beta'_1,\beta'_2}$ is a bimodule of $A_{\alpha'_1,\alpha'_2}$.
\end{thm}
\begin{proof}
We prove only one axiom. The others being proved similarly. For any
$x,y\in A$ and $v\in V$, we have
\begin{align*}
   & \widetilde{l_\prec}(\{\a_2\a'_2(x),\a_1\a'_1(y)\}_{\a'_1,\a'_2}\b_2\b'_2(v)\\=& \widetilde{l_\prec}(\{\a_2\a'_2\a'_1(x),\a_1\a'_1\a'_2(y)\}\b_2\b_{2}^{'2}(v)\\=& l_\prec(\{\a_2\a'_2\a_{1}^{'2}(x),\a_1\a_{1}^{'2}\a'_2(y)\}\b_{2}^{'2}(v)\\
    =&l_\ast(\a_1\a_2\a_{1}^{'2}\a'_2(x))l_\prec(\a_1\a_{1}^{'2}\a'_2(y))\b_{2}^{'2}(v)-l_\prec(\a_1\a_2\a_{1}^{'2}\a'_{2}(y))l_\ast(\a_1\a_1^{'2}\a'_{2}(x))\b_2^{'2}(v)~(by~(\ref{1001}))\\
    =&\widetilde{l_\ast}(\a_1\a'_{1}\a_2\a'_{2}(x))\widetilde{l_\prec}(\a_1\a'_{1}(y))v-\widetilde{l_\prec}(\a_1\a'_{1}\a_2\a'_{2}(y))l_\ast(\a_1\a'_{1}(x))v
\end{align*}
\end{proof}
Taking $\alpha'_1=\alpha^{p_1}_{1},~\alpha'_2=\alpha^{p_2}_{2}$ and $\beta'_1=\beta^{q_1}_{1},~\beta'_2=\beta^{q_2}_{2}$ leads to the following statement:
\begin{cor}
Let $(A,\prec, \succ,\ast,\alpha_1,\a_2)$ be a noncommutative BiHom-pre-Poisson algebra
$(l_{\prec},r_{\prec},l_{\succ},r_{\succ},l_{\ast},r_{\ast},\beta_1,\b_2,V)$ a bimodule of
$A$. Then $V_{\beta_1^{q_1},\beta_2^{q_2}}$ is a bimodule of $A_{\alpha_1^{p_1},\alpha_1^{p_2}}$ for any
nonnegative integers $p_1,~p_2,~q_1$ and $q_2$.
\end{cor}

Let ($ l_{\prec}, r_{\prec}, l_{\succ}, r_{\succ},l_{\diamond}, r_{\diamond}, \beta_{1}, \beta_{2}, V$) be a bimodule of a noncommutative BiHom-pre-Poisson algebra $(A, \prec, \succ,\diamond, \alpha_{1}, \alpha_{2})$ and let $l_{\prec}^{\ast}, r_{\prec}^{\ast}, l_{\succ}^{\ast}, r_{\succ}^{\ast},l_{\diamond}^{\ast}, r_{\diamond}^{\ast}:A\rightarrow gl(V^{\ast}),$ furthermore $\alpha_1^{\ast},\alpha_{2}^{\ast}:A^{\ast}\rightarrow A^{\ast},~~\beta_{1}^{\ast},\beta_{2}^{\ast}:V^{\ast}\rightarrow V^{\ast}$ be the dual maps of respectively $\alpha_1,\alpha_2,\beta_1$ and $\beta_2$ such that
$$\begin{array}{llllllll}
  \langle l_{\prec}^{\ast}(x)u^{\ast},v\rangle =\langle u^{\ast},l_{\prec}(x)v\rangle,&& \langle r^{\ast}_{\prec}(x)u^{\ast},v\rangle =\langle u^{\ast},r_{\prec}(x)v\rangle\\
   \langle l_{\succ}^{\ast}(x)u^{\ast},v\rangle =\langle u^{\ast},l_{\succ}(x)v\rangle,&& \langle r^{\ast}_{\succ}(x)u^{\ast},v\rangle =\langle u^{\ast},r_{\succ}(x)v\rangle\\
    \langle l_{\diamond}^{\ast}(x)u^{\ast},v\rangle =\langle u^{\ast},l_{\diamond}(x)v\rangle,&& \langle r^{\ast}_{\diamond}(x)u^{\ast},v\rangle =\langle u^{\ast},r_{\diamond}(x)v\rangle\\
    \alpha_{1}^{\ast}(x^{\ast}(y))=x^{\ast}(\alpha_{1}(y)),&& \alpha_{2}^{\ast}(x^{\ast}(y))=x^{\ast}(\alpha_{2}(y))\\
     \beta_{1}^{\ast}(u^{\ast}(v))=u^{\ast}(\beta_{1}(v)),&& \beta_{2}^{\ast}(u^{\ast}(v))=u^{\ast}(\beta_{2}(v))
\end{array}$$
\begin{prop}
Let ($ l_{\prec}, r_{\prec}, l_{\succ}, r_{\succ},l_{\diamond}, r_{\diamond}, \beta_{1}, \beta_{2}, V$) be a bimodule of a noncommutative BiHom-pre-Poisson algebra $(A, \prec, \succ,\diamond, \alpha_{1}, \alpha_{2})$. Then ($ l_{\prec}^{\ast}, r_{\prec}^{\ast}, l_{\succ}^{\ast}, r_{\succ}^{\ast},l_{\diamond}^{\ast}, r_{\diamond}^{\ast}, \beta_{1}^{\ast}, \beta_{2}^{\ast}, V^{\ast}$) is a bimodule of
$(A, \prec, \succ,\diamond,\alpha_{1}, \alpha_{2})$ provided that
\begin{eqnarray}
&&\b_2(l_\diamond(\a_2(x)\cdot\a_1(y))u=\b_1 l_\diamond(\a_2(x))r_\succ(a_2(y))u+l_\diamond(\a_1(y))l_\prec(\a_1\a_2(x))u,\\
&&\b_1 l_\cdot(\a_2(y))r_\diamond(\a_2(x))u=\b_2(l_\succ(\a_2(y)\diamond\a_1(x)))u+\b_1 r_\diamond(x)l_\prec(\a_1\a_2(y))u,\\
&&\b_2r_\cdot(\a_1(y))r_\diamond(\a_2(x))u=\b_2r_\diamond(\a_1(x))r_\succ(\a_2(y))u+\b_1\b_2(r_\prec(\a_1(y)\diamond x))u,\\
&&\b_1\rho(\a_1\a_2(y))l_\succ(\a_2(x))u=\b_1l_\succ(x)l_\diamond(\a_1\a_2^{2}(y))u-\b_1\b_2(l_\succ(\a_2^{2}(y)\diamond z))u,\\
&&\b_2(r_\succ(\{\a_1\a_2(x),\a_1(y)\}))u=r_\succ(\a_1(y))l_\diamond(\a_1\a_2^{2}(x))u-l_\diamond(\a_2^{2}(y))r_\succ(\a_1\a_2(y))u,\\
&&-\b_1^{2}\rho(\a_2(y))l_\succ(\a_2(x))u=\b_1\b_2^{2}(r_{\diamond}(x\succ\a_1(y)))u-\b_2^{2}r_\diamond(x)r_\succ(\a_1\a_2(y))u,\\
&&\b_2(l_\prec(\{\a_2(x),\a_1(y)\})u=l_\prec(\a_1(y))l_\diamond(\a_1\a_2(x))u-l_\diamond(\a_1(x))l_\prec(\a_1\a_2(y))u,\\
&&\b_1\rho(\a_2(y))r_\prec(\a_2(x))u=\b_1r_\prec(x)l_\diamond(\a_1\a_2(y))u-\b_1\b_2(r_\prec(\a_1(y)\diamond x))u,\\
&&-\b_1\rho(\a_2(y))r_\prec(\a_2(x))u=\b_1\b_2(r_\diamond(\a_1(y)\prec x))u-\b_1 r_\diamond(x)l_{\prec}(\a_1\a_2(y))(v),
\end{eqnarray}
for all $x,y\in A$ and $u\in V$.
\end{prop}
\begin{proof}
Straightforward.
\end{proof}
\begin{thm}
Let $(A, \prec_{A}, \succ_{A},\ast_A, \alpha_1,\a_2)$ and $(B, \prec_{B}, \succ_{B}, \ast_B,\beta_1,\b_2)$
 be two noncommutative BiHom-pre-Poisson algebra. Suppose that there are linear maps
$ l_{\prec_{A}}, r_{\prec_{A}},  l_{\succ_{A}},  r_{\succ_{A}},l_{\ast_A},r_{\ast_A} : A \rightarrow gl(B),$
and $ l_{\prec_{B}},   r_{\prec_{B}},  l_{\succ_{B}},  r_{\succ_{B}},l_{\ast_B},r_{\ast_B} : B \rightarrow gl(A)$
such that
$A\bowtie^{l_{\ast_A},r_{\ast_A},\beta_1,\beta_2}_{l_{\ast_B},r_{\ast_B},\alpha_1,\alpha_2}B$ is a matched pair of BiHom-pre-Lie algebras and $A\bowtie^{l_{\prec_A},r_{\prec_A},l_{\succ_A},r_{\succ_A},\beta_1,\beta_2}_{l_{\prec_B},r_{\prec_B},l_{\succ_B},r_{\succ_B},\alpha_1,\alpha_2}B$ is a matched pair of BiHom-dendriform algebra and for all $x, y \in A,~ a, b \in
B$, the following equalities hold:
       \begin{eqnarray}
&&-l_{\prec_A}(\rho_B(\b_2(a))\a_1(x))\b_2(b)+\rho_A(\a_2(x))\b_1(a)\prec_B\b_2(b)\nonumber\\&=&l_{\ast_A}(\a_1\a_2(x))(\b_1(a)\prec_B b)-\b_1\b_2(a)\prec_B(l_{\ast_A}(\a_1(x))b)\nonumber\\&&-r_{\prec_A}(r_{\ast_B}(b)\a_1(x))\b_1\b_2(a),\label{301}\\
&&l_{\prec_A}(\rho_B(\b_2(a))\a_1(x))\b_2(b)-(\rho_A(\a_2(x))\b_1(a))\prec_B\b_2(b)\nonumber\\
&=&\b_1\b_2(a)\ast_B\rho(\a_1(x))b+r_{\ast_A}(r_{\prec_B}(b)\a_1(x))\b_1\b_2(a)\nonumber\\&&-l_{\prec_A}(\a_1\a_2(x))(\b_1(a)\ast_B b),\label{302}\\
&&r_{\prec_A}(\a_2(x))(\{\b_2(a),\b_1(b)\}_B)=\b_1\b_2(a)\ast_B(r_{\prec_A}(x)\b_1(b))\nonumber\\&&+r_{\ast_A}(l_{\prec_B}(\b_1(b))x)\b_1\b_2(a)-l_{\prec_A}(\a_1\a_2(x))(\b_1(a)\ast_B b),\label{303}\\
&&l_{\succ_A}(\a_2(x))\{\b_1\b_2(a),\b_1(b)\}=\b_1\b_2^{2}(a)\ast_B(l_{\prec_A}(x)\b_1(b))\nonumber\\&&+r_{\ast_A}(r_{\prec_B}(\b_1(b))x)\b_1\b_2^{2}(a)-(r_{\ast_A}(x)\b_2^{2}(a))\prec_B\b_1\b_2(b)\nonumber\\&&-l_{\prec_A}(l_{\ast_A}(\b_2^{2}(a))x)\b_1\b_2(b),\label{304}\\
&&\b_2(a)\prec_B(\rho_A(\a_1\a_2(x)))\b_1(b)-r_{\prec_A}(\rho_B(\b_2(b))\a_1^{2}(x))\b_2(a)\nonumber\\&=&l_{\ast_A}(\a_1\a_2^{2}(x))(a\succ_B\b_1(b))-(l_{\ast_A}(\a_2^{2}(x))a)\prec_A\b_1\b_2(b)\nonumber\\&&-l_{\succ_A}(r_{\ast_B}(a)\a_2^{2}(x))\b_1\b_2(b),\label{305}\\
&&-\b_2(a)\succ_B(\rho(\a_2(x))\a_1^{2}(b))+r_{\succ_A}(\rho(\b_1\b_2(b))\a_1(x))\b_2(a)\nonumber\\&=&\b_1\b_2^{2}(b)\ast_B(r_{\succ_A}(\a_1(x))b)+r_{\ast_A}(l_{\succ_B}(a)\a_1(x))\nonumber\\&&-r_{\succ_A}(\a_1\a_2(x))(\b_2^{2}(b)\ast_B a),\label{306}\\
&&(l_{\cdot_B}(\a_2(x))\b_1(a))\ast_B\b_{2}(b)+l_{\ast_A}(r_{\cdot_B}(\b_1(a))\a_2(x))\b_2(b)\nonumber\\&=&(l_{\ast_A}(\a_2(x))\b_1(b))\succ_B\b_2(a)+l_{\succ_A}(r_{\ast_B}(\b_1(b))\a_2(x))\b_2(a)\nonumber\\&&+l_{\prec_A}(\a_1\a_2(x))(\b_1(a)\ast_B b),\label{307}\\
&&l_{\ast_A}(r_{\cdot_A}(\a_1(x))\b_2(a))\b_2(b)+(r_{\cdot_A}(\a_1(x))\b_2(a))\ast_B\b_2(b)\nonumber\\&=&r_{\succ_A}(\a_2(x))(\b_2(a)\ast_B\b_1(b))+\b_1\b_2(a)\prec_B(l_{\ast_A}(\a_1(x))b)\nonumber\\&&+r_{\prec_A}(r_{\ast_B}(b)\a_1(x))\b_1\b_2(a),\label{308}\\
&&r_{\ast_A}(\a_2(x))(\b_2(a)\cdot_B\b_1(b))=(r_{\ast_A}(\a_1(x))\b_2(a))\succ_B\b_2(b)\nonumber\\&&+l_{\succ_A}(l_{\ast_B}(\b_2(a))\a_1(x))\b_2(b)+\b_1\b_2(a)\prec_B(r_{\ast_A}(x)\b_1(b))\nonumber\\&&+r_{\prec_A}(l_{\ast_B}(\b_2(a))\a_1(x))\b_2(b)+\b_1\b_2(a)\prec_B(r_{\ast_A}(x)\b_1(b))\nonumber\\&&+r_{\prec_A}(l_{\ast_B}(\b_1(b))x)\b_1\b_2(x),\label{309}\\
&&-l_{\prec_B}(\rho_A(\a_2(x))\b_1(a))\a_2(y)+\rho_B(\b_2(a))\a_1(x)\prec_A\a_2(y)\nonumber\\&=&l_{\ast_B}(\b_1\b_2(a))(\a_1(x)\prec_A y)-\a_1\a_2(x)\prec_A(l_{\ast_B}(\b_1(a))y)\nonumber\\&&-r_{\prec_B}(r_{\ast_A}(y)\b_1(a))\a_1\a_2(x),\label{310}\\
&&l_{\prec_B}(\rho_A(\a_2(x))\b_1(a))\a_2(y)-(\rho_B(\b_2(a))\a_1(x))\prec_A\a_2(y)\nonumber\\
&=&\a_1\a_2(x)\ast_A\rho(\b_1(a))y+r_{\ast_B}(r_{\prec_A}(y)\b_1(a))\a_1\a_2(x)\nonumber\\&&-l_{\prec_B}(\b_1\b_2(a))(\a_1(x)\ast_A y),\label{311}\\
&&r_{\prec_B}(\b_2(a))(\{\a_2(x),\a_1(y)\}_A)=\a_1\a_2(x)\ast_A(r_{\prec_B}(a)\a_1(y))\nonumber\\&&+r_{\ast_B}(l_{\prec_A}(\a_1(y))a)\a_1\a_2(x)-l_{\prec_B}(\b_1\b_2(a))(\a_1(x)\ast_A y),\label{312}\\
&&l_{\succ_B}(\b_2(a))\{\a_1\a_2(x),\a_1(y)\}=\a_1\a_2^{2}(x)\ast_A(l_{\prec_B}(x)\a_1(y))\nonumber\\&&+r_{\ast_B}(r_{\prec_A}(\a_1(y))a)\a_1\a_2^{2}(x)-(r_{\ast_B}(a)\a_2^{2}(x))\prec_A\a_1\a_2(y)\nonumber\\&&-l_{\prec_B}(l_{\ast_B}(\a_2^{2}(x))a)\a_1\a_2(y),\label{313}\\
&&\a_2(x)\prec_A(\rho_B(\b_1\b_2(a)))\a_1(y)-r_{\prec_B}(\rho_A(\a_2(y))\b_1^{2}(a))\a_2(x)\nonumber\\&=&l_{\ast_B}(\b_1\b_2^{2}(x))(a\succ_B\a_1(b))-(l_{\ast_A}(\b_2^{2}(a))x)\prec_B\a_1\a_2(y)\nonumber\\&&-l_{\succ_B}(r_{\ast_A}(x)\b_2^{2}(a))\a_1\a_2(y),\label{314}\\
&&-\a_2(x)\succ_A(\rho(\b_2(a))\b_1^{2}(y))+r_{\succ_B}(\rho(\a_1\a_2(y))\b_1(a))\a_2(x)\nonumber\\&=&\a_1\a_2^{2}(y)\ast_A(r_{\succ_B}(\b_1(a))y)+r_{\ast_B}(l_{\succ_A}(x)\b_1(a))\nonumber\\&&-r_{\succ_B}(\b_1\b_2(a))(\a_2^{2}(y)\ast_A x),\label{315}\\
&&(l_{\cdot_A}(\b_2(a))\a_1(x))\ast_A\a_{2}(y)+l_{\ast_B}(r_{\cdot_A}(\a_1(x))\b_2(a))\a_2(y)\nonumber\\&=&(l_{\ast_B}(\b_2(a))\a_1(y))\succ_A\a_2(x)+l_{\succ_B}(r_{\ast_A}(\a_1(y))\b_2(a))\a_2(x)\nonumber\\&&+l_{\prec_B}(\b_1\b_2(a))(\a_1(x)\ast_A y),\label{316}\\
&&l_{\ast_B}(r_{\cdot_B}(\b_1(a))\a_2(x))\a_2(y)+(r_{\cdot_B}(\b_1(a))\a_2(x))\ast_A\a_2(y)\nonumber\\&=&r_{\succ_B}(\b_2(a))(\a_2(x)\ast_A\a_1(y))+\a_1\a_2(x)\prec_A(l_{\ast_B}(\b_1(a))y)\nonumber\\&&+r_{\prec_B}(r_{\ast_A}(y)\b_1(a))\a_1\a_2(x),\label{317}\\
&&r_{\ast_B}(\b_2(a))(\a_2(x)\cdot_A\a_1(y))=(r_{\ast_B}(\b_1(a))\a_2(x))\succ_A\a_2(y)\nonumber\\&&+l_{\succ_B}(l_{\ast_A}(\a_2(x))\b_1(a))\a_2(y)+\a_1\a_2(x)\prec_A(r_{\ast_B}(a)\a_1(y))\nonumber\\&&+r_{\prec_B}(l_{\ast_A}(\a_2(x))\b_1(a))\a_2(y)+\a_1\a_2(x)\prec_A(r_{\ast_B}(a)\a_1(y))\nonumber\\&&+r_{\prec_B}(l_{\ast_A}(\a_1(y))a)\a_1\a_2(a),\label{318}
\end{eqnarray}
where
$$\begin{array}{lllllll} x \cdot_A y = x \prec_A y + x \succ_A y,~ ~~~~~l_{\cdot_A} = l_{\prec_A} +
l_{\succ_A},~~~~~~r_{\cdot_A} = r_{\prec_A} + r_{\succ_A},
\\
a \cdot_B b = a \prec_B b + a \succ_B b,~~~~~~ l_{\cdot_B} = l_{\prec_B} +
l_{\succ_B}, ~~~~~~r_{\cdot_B} = r_{\prec_B} + r_{\succ_B},
\\
 \{\a_2(x),\a_1(y)\}_A = \a_2(x)
\ast_A \a_1(y) - \a_2(y) \ast_A \a_1(x),\\
\{\b_2(a),\b_1(b)\}_B = \b_2(a)
\ast_B \b_1(b) - \b_2(b) \ast_B \b_1(a),\\ (\rho_A\circ\a_2)\b_1 = (l_{\ast_A}\circ\a_2)\b_1 - (r_{\ast_A}\circ \a_1)\b_2,\\
(\rho_B\circ \b_2)\a_1 = (l_{\ast_B}\circ\b_2)\a_1 - (r_{\ast_B}\circ \b_1)\a_2.
\end{array}$$
Then $(A,B,l_{\prec_A},r_{\prec_A},l_{\succ_A},r_{\succ_A},l_{\ast_A},r_{\ast_A},\beta_1,\beta_2,l_{\prec_B},r_{\prec_B},l_{\succ_B},r_{\succ_B},l_{\ast_B},r_{\ast_B},\alpha_1,\alpha_2)$ is called a matched pair of noncommutative BiHom-pre-Poisson algebras. In this case, there exists a noncommutative BiHom-pre-Poisson algebra
structure on the direct sum $ A \oplus B $ of the underlying vector spaces of
 $ A $ and $ B $ given by
\begin{eqnarray*}
(x + a) \prec ( y + b ) &:=& (x \prec_{A} y + r_{\prec_{B}}(b)x + l_{\prec_{B}}(a)y)
+(l_{\prec_{A}}(x)b + r_{\prec_{A}}(y)a + a \prec_{B} b ), \cr
(x + a) \succ ( y + b ) &:=& (x \succ_{A} y + r_{\succ_{B}}(b)x + l_{\succ_{B}}(a)y)
+ (l_{\succ_{A}}(x)b + r_{\succ_{A}}(y)a + a \succ_{B} b )\cr
(x + a) \ast ( y + b ) &:=& (x \ast_{A} y + r_{\ast_{B}}(b)x + l_{\ast_{B}}(a)y)
+(l_{\ast_{A}}(x)b + r_{\ast_{A}}(y)a + a \ast_{B} b ),\cr
(\a_1\oplus\b_1)(x+a)&:=&\a_1(x)+\b_1(a),\cr
(\a_2\oplus\b_2)(x+a)&:=&\a_2(x)+\b_2(a),
\end{eqnarray*}
for any $ x, y \in A, ~a, b \in B $.
            \end{thm}
\begin{proof}
It is obtained in a similar way as for Theorem \ref{matched ass}.

\end{proof}

Let $ A \bowtie^{l_{\prec_{A}}, r_{\prec_{A}},
l_{\succ_{A}}, r_{\succ_{A}},l_{\ast_{A}}, r_{\ast_{A}}, \beta_1,\b_2}_{l_{\prec_{B}}, r_{\prec_{B}}, l_{\succ_{B}},
r_{\succ_{B}},l_{\ast_{B}},
r_{\ast_{B}}, \alpha_1,\a_2} B$ denote this noncommutative BiHom-pre-Poisson algebra.
\begin{cor}
Let $(A, B, l_{\prec_{A}}, r_{\prec_{A}}, l_{\succ_{A}}, r_{\succ_{A}},l_{\ast_{A}}, r_{\ast_{A}}, \beta_1,\b_2,
 l_{\prec_{B}}, r_{\prec_{B}}, l_{\succ_{B}}, r_{\succ_{B}},l_{\ast_{B}}, r_{\ast_{B}}, \alpha_1,\a_2) $ be a matched pair of regular noncommutative BiHom-pre-Poisson algebras $(A, \prec_{A}, \succ_{A},\ast_A, \alpha_1,\a_2)$ and $(B, \prec_{B}, \succ_{B}, \ast_B,\beta_1,\b_2)$.
Then, $(A, B, l_{\prec_{A}} + l_{\succ_{A}}, r_{\prec_{A}} + r_{\succ_{A}},l_{\ast_{A}}-(r_{\ast_{A}}\circ\a_1\a_2^{-1})\b_1^{-1}\b_2,\b_1,\b_2,
l_{\prec_{B}} + l_{\succ_{B}},  r_{\prec_{B}} + r_{\succ_{B}} ,l_{\ast_{B}}-(r_{\ast_{B}}\circ\b_1\b_2^{-1})\a_1^{-1}\a_2,\a_1,\a_2)$ is a matched pair of the associated
noncommutative BiHom-Poisson algebras $(A,\cdot_{A},\{\cdot,\cdot\}_A, \alpha_1,\a_2)$ and $(B, \cdot_{B},\{\cdot,\cdot\}_B,\beta_1,\b_2)$.
\end{cor}
\begin{proof} Let $(A, B, l_{\prec_{A}}, r_{\prec_{A}}, l_{\succ_{A}}, r_{\succ_{A}},l_{\ast_{A}}, r_{\ast_{A}}, \beta_1,\b_2,
 l_{\prec_{B}}, r_{\prec_{B}}, l_{\succ_{B}}, r_{\succ_{B}},l_{\ast_{B}}, r_{\ast_{B}}, \alpha_1,\a_2) $ be a matched pair of regular noncommutative BiHom-pre-Poisson algebras $(A, \prec_{A}, \succ_{A},\ast_A, \alpha_1,\a_2)$ and $(B, \prec_{B}, \succ_{B}, \ast_B,\beta_1,\b_2)$. Then by Proposition \ref{Matchedd1} and Proposition \ref{Matchedd2}, $(A, B, l_{\prec_{A}} + l_{\succ_{A}}, r_{\prec_{A}} + r_{\succ_{A}},\b_1,\b_2,
l_{\prec_{B}} + l_{\succ_{B}},  r_{\prec_{B}} + r_{\succ_{B}} ,\a_1,\a_2)$ is a matched pair of the associated
 BiHom-associative algebras $(A,\cdot_{A}, \alpha_1,\a_2)$ and $(B, \cdot_{B},\beta_1,\b_2)$ and $(A, B, l_{\ast_{A}}-(r_{\ast_{A}}\circ\a_1\a_2^{-1})\b_1^{-1}\b_2,\b_1,\b_2,
l_{\ast_{B}}-(r_{\ast_{B}}\circ\b_1\b_2^{-1})\a_1^{-1}\a_2,\a_1,\a_2)$ is a matched pair of the associated
BiHom-Lie algebras $(A,\{\cdot,\cdot\}_A, \alpha_1,\a_2)$ and $(B, \{\cdot,\cdot\}_B,\beta_1,\b_2)$. Besides, in view of Corollary \ref{coroo}, the linear maps $l_{\prec_{A}} + l_{\succ_{A}}, r_{\prec_{A}} + r_{\succ_{A}},l_{\ast_{A}}-(r_{\ast_{A}}\circ\a_1\a_2^{-1})\b_1^{-1}\b_2:A\rightarrow gl(B)$ and  $l_{\prec_{B}} + l_{\succ_{B}},  r_{\prec_{B}} + r_{\succ_{B}},l_{\ast_{B}}-(r_{\ast_{B}}\circ\b_1\b_2^{-1})\a_1^{-1}\a_2:B\rightarrow gl(A)$ are a representations of the underlying noncommutative BiHom-Poisson algebras $(A,\cdot_{A},\{\cdot,\cdot\}_A, \alpha_1,\a_2)$ and $(B, \cdot_{B},\{\cdot,\cdot\}_B,\beta_1,\b_2)$, respectively. Therefore, (\ref{101})-(\ref{102}) are equivalents to (\ref{301})-(\ref{309}) and (\ref{103})-(\ref{104}) are equivalents to (\ref{310})-(\ref{318}).
\end{proof}
\section{$\mathcal{O}$-operators of noncommutative BiHom-Poisson algebras}
In this section we introduce the notions of an $\mathcal{O}$-operator of noncommutative BiHom-
Poisson algebras and we give some related
properties.
\begin{defn}
Let $(A, \cdot, \alpha_1,\a_2)$ be a BiHom-associative algebra and $(l, r, \beta_1,\b_2, V)$ be a bimodule of $A$. Then, a  linear map $ T : V \rightarrow A $
is called an $\mathcal{O}$-operator associated to $(l, r, \beta_1,\b_2, V)$,  if $ T $ satisfies
\begin{eqnarray*}
\alpha_1 T= T\beta_1,~\alpha_2 T= T\beta_2 \mbox{ and } T(u)\cdot T(v) = T(l(T(u))v + r(T(v))u) \mbox { for all } u, v \in V.
\end{eqnarray*}
\end{defn}
\begin{lem}\label{lem silv double}\cite{double}
Let $(A, \cdot, \alpha_1,\a_2)$ be a BiHom-associative algebra, and let $(l, r, \beta_1,\b_2, V) $ be a bimodule.
Let $ T : V \rightarrow A $ be an $ \mathcal{O}$-operator associated to $(l, r, \beta_1,\b_2, V)$. Then, there exists a BiHom-dendriform
algebra structure on $ V $ given by
\begin{eqnarray*}
 u \succ v = l(T(u))v , ~ u \prec v = r(T(v))u
\end{eqnarray*}
for all $u, v \in V$.
\end{lem}
Now we recall the definition of an $ \mathcal{O}$-operator on a BiHom-Lie algebra associated to a
given representation, which generalize the Rota-Baxter operator of weight $0$ introduced
in \cite{luimakhlouf}.
\begin{defn}
Let $(A, \{\cdot,\cdot\}, \alpha_1,\a_2)$ be a BiHom-Lie algebra, and let $(\rho, \beta_1,\b_2,V)$ be a representation of $A$. Then, a  linear map $ T : V \rightarrow A $
is called an $ \mathcal{O} $-operator associated to $(\rho, \beta_1,\b_2, V)$,  if $ T $ satisfies
\begin{eqnarray*}
\alpha_1 T= T\beta_1,~\alpha_2 T= T\beta_2 \mbox{ and } \{T(u), T(v)\} = T(\rho(T(u))v - \rho(T(\b_1^{-1}\b_2(v)))\b_1\b_2^{-1}(u)) \mbox { for all } u, v \in V.
\end{eqnarray*}
\end{defn}
\begin{ex}
An $\mathcal{O}$-operator on a BiHom-Lie algebra $(A,\{\cdot,\cdot\},\alpha_1,\a_2)$ with respect to the adjoint representation is called a Rota-Baxter operator on $A$.
\end{ex}
\begin{lem}\label{lemm2}
Let $T:V\rightarrow A$ be an $\mathcal{O}$-operator on a BiHom-Lie algebra $(A,\{\cdot,\cdot\},\alpha_1,\a_2)$ with respect to a representation $(\rho, \beta_1,\b_2,V)$. Define a multiplication $\ast$ on $V$ by
\begin{equation}
    u\ast v=\rho(T(u))v,~~\forall u,v\in V.
\end{equation}
Then $(V,\ast,\alpha_1,\a_2)$ is a BiHom-pre-Lie algebra.
\end{lem}
\begin{defn}
Let $(A,\cdot, \{\cdot,\cdot\}, \alpha_1,\a_2)$ be a noncommutative BiHom-Poisson algebra, and let $(l,r,\rho, \beta_1,\b_2,V)$ be a representation of $A$.
A linear operator $T:V\rightarrow A$ is called an $\mathcal{O}$-operator on $A$ if $T$ is both an $\mathcal{O}$-operator on the BiHom-associative algebra $(A,\cdot,\alpha_1,\a_2)$ and an $\mathcal{O}$-operator on the BiHom-Lie algebra $(A,\{\cdot,\cdot\},\alpha_1,\a_2)$.
\end{defn}
\begin{ex}
An $\mathcal{O}$-operator on a noncommutative BiHom-Poisson algebra $(A,\cdot,\{\cdot,\cdot\},\alpha_1,\a_2)$ with respect the regular representation is called a Rota-Baxter operator on $A$.
\end{ex}
\begin{thm}
Let $(A,\cdot,\{\cdot,\cdot\},\alpha_1,\a_2)$ be a noncommutative BiHom-Poisson algebra and $T:V\rightarrow A$ an $\mathcal{O}$-operator on $A$ with respect to the representation
$(l,r,\rho,\beta_1,\b_2,V)$. Define new operations $\prec,\succ$ and $\ast$ on $V$ by \begin{equation}\label{lara2}u\prec v=l(T(u))v,~u\succ v=r(T(v))u,~u\ast v=\rho(T(u))v.\end{equation}
Then $(V,\prec,\succ,\ast,\alpha_1,\a_2)$ is a noncommutative BiHom-pre-Poisson algebra. Moreover, $T(V)=\{T(v);~v\in V\}\subset A$ is a subalgebra of $A$ and there is an induced noncommutative BiHom-pre-Poisson algebra structure on $T(V)$ given by
\begin{equation}\label{lara1}
T(u)\prec T(v)=T(u\prec v),~~T(u)\succ T(v)=T(u\succ v),~~T(u)\ast T(v)=T(u\ast v),\end{equation}
for all $u,v\in V$.
\end{thm}
\begin{proof}
By Lemma \ref{lem silv double} and Lemma \ref{lemm2}, we deduce that $(A,\prec,\succ,\alpha_1,\a_2)$ is a BiHom-dendriform algebra and $(A,\ast,\alpha_1,\a_2)$ is a BiHom-pre-Lie algebra. Now, we  prove  only  the  axiom (\ref{eq:pre-Poisson 1}). The  other being  proved  similarly, for any $x,y,z\in V$ we have
\begin{align*}
    &(\b_2(x)\ast\b_1(y)-\b_2(y)\ast\b_1(x))\prec\b_2(z)\\&-\b_1\b_2(x)\ast(\b_1(y)\prec z)+\b_1\b_2(y)\prec(\b_1(x)\ast z)\\=&(\rho(T(\b_2(x))\b_2(y)-\rho(T(\b_2(y))\b_1(x))\prec\b_2(z)\\&-\rho(T(\b_1\b_2(x)))(\b_1(y)\prec z)+l(T(\b_1\b_2(y)))(\b_1(x)\ast z)\\
    =&l(T(\rho(T(\b_2(x))\b_1(y)\\&-\rho(T(\b_2(y)))\b_1(x))\b_2(z)-\rho(T(\b_1\b_2(x)))l(\b_1(y))z\\&+l(T(\b_1\b_2(y)))\rho(\b_1(x))z\\
    =&l(\{T(\b_2(x),T(\b_1(y))\})\b_2(z)-\rho(T(\b_1\b_2(x)))l(\b_1(y))z\\&+l(T(\b_1\b_2(y)))\rho(\b_1(x))z=0~~~(by~(\ref{isma1.1})).
\end{align*}
Therefore, $(V,\prec,\succ,\ast,\alpha_1,\a_2)$ is a BiHom-pre-Posson algebra.\\The rest is straightforward.
\end{proof}
\begin{cor}
Let $(A,\cdot,\{\cdot,\cdot\},\alpha_1,\a_2)$ be a noncommutative BiHom-Poisson algebra. Then there is a noncommutative BiHom-pre-Poisson algebra structure on $A$ such that its sub-adjacent noncommutative BiHom-Poisson algebra is exactly $(A,\cdot,\{\cdot,\cdot\},\alpha_1,\a_2)$ if and only if there exists an invertible $\mathcal{O}$-operator on $(A,\cdot,\{\cdot,\cdot\},\alpha_1,\a_2).$
\end{cor}
\begin{proof}
Suppose that there exists an invertible $\mathcal{O}$-operator $T:V\rightarrow A$ associated to the representation $(l,r,\rho,\b_1,\b_2, V)$, then the compatible noncommutative BiHom-pre-Poisson algebra structure on $A$, for all $x,y\in A$ is given by
$$x\prec y=T(l(x)T^{-1}(y)),\;\; x\succ y=T(r(y)T^{-1}(x)),\;\;x\ast y=T(\rho(x)T^{-1}(y))~\forall x,y\in A.$$

Conversely, let $(A,\prec,\succ,\ast,\alpha_1,\a_2)$ be a noncommutative BiHHom-pre-Poisson algebra and $(A,\cdot,\{\cdot,\cdot\},\alpha_1,\a_2)$ the sub-adjacent noncommutative BiHom-Poisson algebra. Then the identity map $id$ is an $\mathcal{O}$-operator on $A$ with respect to the regular representation $(L_\prec,R_\succ, ad,\alpha_1,\a_2,A)$.

\end{proof}
\begin{ex}
  Let $(A,\cdot,\{\cdot,\cdot\},\alpha_1,\a_2)$ be a noncommutative BiHom-Poisson algebra and $R:A\longrightarrow
A$ a Rota-Baxter operator. Define new operations on $A$ by
$$x\prec y=R(x)\cdot y,\quad x\succ y=x\cdot R(y),\quad x\ast y=\{R(x),y\}.$$
Then $(A,\prec,\succ,\ast,\alpha_1,\a_2)$ is a noncommutative BiHom-pre-Poisson algebra and $R$ is a homomorphism from the sub-adjacent noncommutative BiHom-Poisson algebra $(A,\cdot',\{\cdot,\cdot\}',\alpha_1,\a_2)$ to $(A,\cdot,\{\cdot,\cdot\},\alpha_1,\a_2)$, where $x\cdot' y=x\prec y+x\succ y$ and $\{x,y\}'=x\ast y-\a_1^{-1}\a_2(y)\ast \a_1\a_2^{-1}(x)$.\\
\end{ex}

The inverse relation existing between a noncommutative BiHom-pre-Poisson
algebra and noncommutative BiHom-Poisson
algebra, as illustrated by the following diagram:
$$
\xymatrix{ \mbox{BiHom-associative alg+BiHom-Lie alg} \ar[rr] \ar[dd]^{\{R(x),y\}}_{R(x)\cdot y,~~~~x\cdot R(y)} && \mbox{BiHom-Poisson alg} \ar[dd]^{\{R(x),y\}}_{R(x)\cdot y,~~~~x\cdot R(y)}\\
&& \\
\mbox{BiHom-dendriform alg+ BiHom-pre-Lie alg} \ar[rr] && \mbox{BiHom-pre-Poisson alg.}
}
$$


\begin{thebibliography}{00}
\bibitem{Elkadri} Abdaoui, K., Ammar, F., Makhlouf, A.: Constructions and cohomology of Hom-Lie color algebras,
Comm. Algebra. $\bf 43$(11), 4581-4612 (2015).
\bibitem{hadimi} Adimi, H., Amri, H., Mabrouk, S., Makhlouf, A.: (Non-BiHom-Commutative) BiHom-Poisson algebras, arXiv:2008.04597vi [math.RA]
\bibitem{A2} Aguiar, M.: Pre-Poisson algebras, Lett. Math. Phys. {\bf 54}, 263-277 (2000)

\bibitem{HnkHndSilvdcbihomfrobalg:AizawaSaito}
Aizawa, N., Sato, H.: $q$-deformation of the Virasoro algebra with central extension, Phys. Lett. B. \textbf{256}, 185-190 (1991) (Hiroshima University preprint, preprint HUPD-9012 (1990))
\bibitem{AmMakh} Ammar, F., Makhlouf, A.: Hom-Lie and Hom-Lie admissible superalgebras, Journal of Algebra. $\bf 324$, 1513-1528 (2010)
\bibitem{attan0}Attan, S.: Some characterizations of color Hom-Poisson algebras, Hacettepe J. of Mathematics and Statistics. $\bf 47$(6), 1552-1563 (2018)
\bibitem{attan1}
Attan, S., Hounnon, H., Kpam\`egan, B.:  Hom-Jordan and Hom-alternative bimodules, Extracta Mathematicae. {\bf 35}, 69-97 (2020)
\bibitem{attan2}
Attan, S.: Structures and bimodules of simple Hom-alternative algebras. arXiv:1908.08711v1 [math.RA]
\bibitem{Ismaiiil} Attan, S., Laraiedh, I.: Structures of BiHom-Poisson algebras, arXiv:2008.04763
\bibitem{Bak1} Bakayoko, I.: Modules over color Hom-Poisson  algebras, J. of Generalized Lie  Theory Appl. \textbf{8}(1), (2014)
\bibitem{CohomologyPA1}
Bao, Y., Ye, Y.: Cohomology structure for a Poisson algebra: I, J. Algebra Appl. {\bf 15}(2) 17pp (2016)

\bibitem{CohomologyPA2}
Y. Bao and Y. Ye,  Cohomology structure for a Poisson algebra: II, \emph{Sci. China Math.} (2019). https://doi.org/10.1007/s11425-019-1591-6.

\bibitem{left symm}
Ben Hassine, A., Chtioui, T., Mabrouk, S., Ncib, O.: Cohomology and linear deformation of BiHom-left-symmetric algebras, arXiv:1907.06979.

\bibitem{HnkHndSilvdcbihomfrobalg:ChaiElinPop}
Chaichian, M., Ellinas, D., Popowicz, Z.: Quantum conformal algebra with central extension, Phys. Lett. B. \textbf{248}, 95-99 (1990)

\bibitem{HnkHndSilvdcbihomfrobalg:ChaiIsLukPopPresn}
Chaichian, M., Isaev, A. P., Lukierski, J., Popowic, Z., Pre\v{s}najder, P.: $q$-deformations of Virasoro algebra and conformal dimensions, Phys. Lett. B. \textbf{262} (1), 32-38 (1991)

\bibitem{HnkHndSilvdcbihomfrobalg:ChaiKuLuk}
Chaichian, M., Kulish, P., Lukierski, J.: $q$-deformed Jacobi identity, $q$-oscillators and $q$-deformed infinite-dimensional algebras, Phys. Lett. B. \textbf{237}, 401-406 (1990)

\bibitem{HnkHndSilvdcbihomfrobalg:ChaiPopPres}
Chaichian, M., Popowicz, Z., Pre\v{s}najder, P.: $q$-Virasoro algebra and its relation to the $q$-deformed KdV system, Phys. Lett. B. \textbf{249}, 63-65 (1990)

\bibitem{HnkHndSilvdcbihomfrobalg:CurtrZachos1}
Curtright, T. L., Zachos, C. K.: Deforming maps for quantum algebras, Phys. Lett. B. \textbf{243}, 237-244 (1990)

\bibitem{HnkHndSilvdcbihomfrobalg:DamKu}
Damaskinsky, E. V., Kulish, P. P.: Deformed oscillators and their applications (in Russian), Zap. Nauch. Semin. LOMI 189, 37-74 (1991) [Engl. transl. in J. Sov. Math. {\bf 62}, 2963-2986 (1992)

\bibitem{HnkHndSilvdcbihomfrobalg:DaskaloyannisGendefVir}
Daskaloyannis, C.: Generalized deformed Virasoro algebras, Modern Phys. Lett. A. \textbf{7}(9), 809-816 (1992)
\bibitem{GRAZIANI}Graziani, G., Makhlouf, A., Menini, C. Panaite, F.: BiHom-Associative Algebras, BiHom-Lie Algebras
and BiHom-Bialgebras, SIGMA Symmetry Integrability Geom. Methods Appl. {\bf 11} (2015)
\bibitem{hls}
Hartwig, J. T., Larsson, D., Silvestrov, S. D.: Deformations of Lie algebras using $\sigma$-derivations, J.
Algebra. \textbf{295}, 314--361 (2006)


\bibitem{double}
Hounkonnou, M. N., Houndedji, G. D.,  Silvestrov, S.: Double constructions of biHom-Frobenius algebras. arXiv:2008.06645 v1 [math.RA], (2020)
\bibitem{HnkHndSilvdcbihomfrobalg:Hu}
Hu, N.: $q$-Witt algebras, $q$-Lie algebras, $q$-holomorph structure and representations, Algebra Colloq. \textbf{6}(1), 51-70 (1999)

\bibitem{HnkHndSilvdcbihomfrobalg:Kassel92}
Kassel, C.: Cyclic homology of differential operators, the virasoro algebra and a $q$-analogue, Comm. Math. Phys. 146 (2), 343-356 (1992)




 \bibitem{Kubo1} Kubo, F.: Finite-dimensional non-commutative Poisson algebras, \emph{J. Pure Appl. Algebra.} {\bf 113}, 307-314 (1996)

\bibitem{Kubo3}Kubo, K., Finite-dimensional simple Leibniz pairs and simple Poisson modules, \emph{Lett. Math. Phys.} {\bf 43}, 21-29 (1998)

\bibitem{Kubo4}Kubo, F.: Non-commutative Poisson algebra structures on affine Kac-Moody aglebras, J. Pure Appl. Algebra. {\bf 126}, 267-286 (1998)


\bibitem{larsson}
Larsson, D., Silvestrov, S. D.:  \emph{Quasi-hom-Lie algebras, central extensions and $2$-cocycle-like identities},
J. Algebra. \textbf{288}, 321--344 (2005)
\bibitem{Baii}
Liu, J., Bai, C., Sheng, Y.: Noncommutative Poisson bialgebras. arXiv:2004.02560
\bibitem{luimakhlouf}
Liu, L., Makhlouf, A., Menini, C., Panaite, F.: BiHom-pre-Lie algebras, BiHom-Leibniz algebras and Rota-Baxter operators on BiHom-Lie algebras, arXiv:1706.00474. (2017)
\bibitem{LiuMakhloufMeninPanaite}Liu, L., Makhlouf, A., Menini, C., Panaite, F.:  \emph{Rota-Baxter operators on BiHom-associative algebras and related structures}, arXiv:1703.07275

\bibitem{HnkHndSilvdcbihomfrobalg:LiuKQuantumCentExt}
Liu, K. Q.: Quantum central extensions, C. R. Math. Rep. Acad. Sci. Canada. \textbf{13}(4), 135-140 (1991)

\bibitem{HnkHndSilvdcbihomfrobalg:LiuKQCharQuantWittAlg}
Liu, K. Q.: Characterizations of the Quantum Witt Algebra, Lett. Math. Phys. \textbf{24}(4), 257-265 (1992)

\bibitem{HnkHndSilvdcbihomfrobalg:LiuKQPhDthesis} Liu, K. Q.: The Quantum Witt Algebra and Quantization of Some Modules over Witt Algebra, PhD Thesis, Department of Mathematics, University of Alberta, Edmonton, Canada. (1992)
\bibitem{XLi} X. Li,  BiHom-Poisson Algebra and Its Application,
\textit{Int. J. Alg.} {\bf 13} (2019) 73-81.
\bibitem{Makhl1}
Makhlouf, A.: Hom-dendriform algebras and Rota-Baxter Hom-algebras, in proceedings of international conferences in Nankai series in pure. In: Bai, C., Guo, L., Loday, J.-L. (eds.) Applied Mathematics and Theoretical Physics, World Scientific, Singapore Vol. 9., (2012), pp. 147-171.
\bibitem{Makhl2} Makhlouf, A., Silvestrov, S. D.: Hom-algebra structures, J. Gen. Lie Theory Appl. \textbf{2}(2), 51--64 (2008)
(Preprints in Mathematical Sciences  2006:10, LUTFMA-5074-2006, Centre for Mathematical Sciences, Department of Mathematics, Lund Institute of Technology, Lund University (2006))

\bibitem{Makhl3} Makhlouf, A., Silvestrov, S.: Hom-Lie Admissible Hom-Coalgebras and Hom-Hopf Algebras, Chapter 17,
in: S. Silvestrov, E. Paal, V. Abramov, A. Stolin (Eds.), Generalized Lie theory in Mathematics, Physics and Beyond, Springer-Verlag, Berlin, Heidelberg, 2009, 189-206


\bibitem{Makhl4}
Makhlouf, A., Yau, D.: Rota-Baxter Hom-Lie admissible algebras, Comm. Alg.  $\bf 23$(3), 1231-1257 (2014)
\bibitem{Van}
Van den Bergh, M.: Double Poisson algebras, Trans. Amer. Math. Soc. {\bf 360}, 5711-5769 (2008)

\bibitem{Xu}
Xu, P.: Noncommutative Poisson algebras, \emph{Amer. J. Math.} \textbf{(116)}, 101-125 (1994)
\bibitem{Yau1} Yau, D.: Hom-algebras and homology, J. Lie Theory. \textbf{19}(2), 409--421 (2009)

\bibitem{Yau2} Yau, D.: A Hom-associative analogue of Hom-Nambu algebras, arXiv:1610.02845v1

\bibitem{Yau3} Yau, D.: Non-commutative Hom-Poisson algebras, arXiv:1010.3408[math.RA]

\bibitem{Yau4} Yau, D.: Hom-bialgebras and comodule Hom-algebras, Int. E. J. Alg. \textbf{8}, 45-64 (2010)

\bibitem{Yau5} Yau, D.: Hom-Malcev, Hom-alternative and Hom-Jordan algebras, Int. Elect. Journ. of Alg. $\bf 11$, 177-217 (2012)

\bibitem{Yau6} Yuan, L.: Hom-Lie color algebras, Comm. Alg. {\bf 40}(2), 575-592 (2012)

\end{thebibliography}
\end{document}